\documentclass[11pt]{amsart}

\usepackage{latexsym}
\usepackage{amssymb}
\usepackage{amsmath}
\usepackage{color}

\newtheorem{theorem}{Theorem}[section]
\newtheorem{lemma}[theorem]{Lemma}
\newtheorem{proposition}[theorem]{Proposition}
\newtheorem{corollary}[theorem]{Corollary}
\newtheorem{definition}[theorem]{Definition}

\newtheorem{remark}[theorem]{Remark}

\newtheorem{question}[theorem]{Question}

\newcommand\nph{\varphi}

\newcommand\rank{\mathop{\rm rank}}

\newcommand\conv{\mathop{\rm conv}}
\newcommand\her{\mathop{\rm her}}

\newcommand\ap{\mathop{\rm ap}}
\newcommand\cp{\mathop{\rm cp}}
\newcommand\thet{\mathop{\rm th}}
\newcommand\fp{\mathop{\rm fp}}
\newcommand\sub{\mathop{\rm sub}}

\newcommand{\cl}[1]{\mathcal{#1}}
\newcommand{\bb}[1]{\mathbb{#1}}

\newcommand\vp{\mathop{\rm vp}}
\newcommand\fvp{\mathop{\rm fvp}}

\newcommand\thab{\mathop{\rm thab}}

\newcommand\Tr{\mathop{\rm Tr}}

\newcommand{\ip}[2]{\ensuremath{\left\langle #1 , #2\right\rangle}}

\begin{document}

\title[Sandwich theorems and capacity bounds]{Sandwich theorems and capacity bounds for non-commutative graphs}



\author[]{G. Boreland}
\address{Mathematical Sciences Research Centre,
Queen's University Belfast, Belfast BT7 1NN, United Kingdom}
\email{gboreland01@qub.ac.uk}

\author[]{I. G. Todorov}
\address{Mathematical Sciences Research Centre,
Queen's University Belfast, Belfast BT7 1NN, United Kingdom,
and
School of Mathematical Sciences, Nankai University, 300071 Tianjin, China}
\email{i.todorov@qub.ac.uk}

\author[]{A. Winter}
\address{ICREA and F\'{\i}sica Te\`{o}rica: Informaci\'{o} i Fenomens Qu\`{a}ntics, Universitat Aut\`{o}noma de Barcelona, ES-08193 Bellaterra, Barcelona, Spain}
\email{andreas.winter@uab.cat}


\date{19 July 2019}

\maketitle

\tableofcontents

\begin{abstract}
We define non-commutative versions of the vertex packing polytope, the theta convex body and the
fractional vertex packing polytope of a graph, and 
establish a quantum version of the Sandwich Theorem of Gr\"{o}tschel, Lov\'{a}sz and Schrijver. 
We define new non-commutative versions of the Lov\'{a}sz number of a graph which 
lead to an upper bound of the zero-error 
capacity of the corresponding quantum channel that can be genuinely better than the one established 
by Duan, Severini and Winter in \cite{dsw}. 
We define 
non-commutative counterparts of widely used classical graph parameters and establish their interrelation. 
\end{abstract}


\section{Introduction}\label{s_intro}

The use of graphs in the study of information theoretic questions has origins in Shannon's seminal paper \cite{shannon},
where he laid the foundations of zero-error information theory. 
With a given information channel $\cl N$, Shannon associated a graph $G_{\cl N}$, called the 
\emph{confusability graph} of the channel, and showed that the zero-error transmission properties of $\cl N$
are captured in their entirety by $G_{\cl N}$. In particular, he defined the 
\emph{zero-error capacity} $c_0(\cl N)$ of the channel $\cl N$ 
as an asymptotic parameter involving the independence numbers of 
the strong powers of $G_{\cl N}$. 
While the information theoretic importance of $c_0(\cl N)$ is easy to appreciate, its computation remains 
a difficult problem, due to the high computational complexity of the independence number. 

An upper bound for $c_0(\cl N)$, computable in polynomial time, was introduced by Lov\'{a}sz in \cite{lo}. 
The parameter $\theta(G)$ of a given graph $G$, defined therein, satisfies the Sandwich Theorem
\begin{equation}\label{eq_clsa}
\alpha(G)\leq \theta(G) \leq \chi_{\rm f}(G^c);
\end{equation}
here $\alpha(G)$ is the independence number of $G$, 
while $\chi_{\rm f}(G^c)$ is the fractional chromatic number of its complement $G^c$. 
The Sandwich Theorem thus provides a simultaneous bound for the outer parameters, which have high computational complexity,
and plays an important role in combinatorial optimisation \cite{gls_book}.

A stronger and more powerful version of the Sandwich Theorem was established in \cite{gls} (see also \cite{knuth}), 
where convex bodies arising from vertex packings of a graph $G$ were introduced 
-- these are the vertex packing polytope $\vp(G)$, the fractional vertex packing polytope $\fvp(G)$ and the 
theta body $\thab(G)$ -- and shown to satisfy the inclusions
\begin{equation}\label{eq_clsacc}
\vp(G) \subseteq \thab(G) \subseteq \fvp(G).
\end{equation}
Inequalities (\ref{eq_clsa}) are then obtained by optimising the trace functional over the chain (\ref{eq_clsacc}).
Its terms are particular examples of \emph{convex corners}, that is, hereditary closed convex subsets of $\bb{R}^d_+$
\cite{cklms, gls_book}.
The importance of the inclusions (\ref{eq_clsacc}) comes from the significance of considering
weighted versions of the trace functional in optimisation problems for graphs \cite{knuth}.

Quantum information analogues of the aforementioned objects and results were initiated in \cite{dsw}, where 
the authors defined a suitable version of the confusability graph of a quantum channel $\Phi$ as 
an operator subsystem (that is, a selfadjoint subspace containing the identity matrix) $\cl S$ 
of the domain $M_d$ of $\Phi$, and showed that it captures the zero-error properties of $\Phi$. 
In particular, they defined the (classical) zero-error capacity of the channel $\Phi$ and showed that 
it depends solely on the operator system $\cl S$. 
A classical graph $G$ gives rise in a canonical fashion to an operator system that remembers $G$ 
\cite{op}. This justifies calling arbitrary operator systems in $M_d$ \emph{non-commutative graphs}, 
and pursuing their study as a non-commutative version of graph theory. 
Advances in this direction were recently made in \cite{lpt}, where classical parameters such as the 
intersection number, the minimum semi-definite rank and the orthogonal rank of the 
complement were lifted to the non-commutative setting and given 
a quantum informational interpretation, 
and in \cite{weaver}, where a version of the Ramsey Theorem was established for operator systems. 
A quantum version of the Lov\'{a}sz number was defined in \cite{dsw}, and shown to be an upper bound of the 
zero-error capacity of quantum channels, computable via semi-definite programming.

The purpose of this paper is two-fold. Firstly, we initiate the study of non-commutative convex corners, establish 
a quantum version of the Sandwich Theorem (\ref{eq_clsacc}) and define a 
new non-commutative version of the classical Lov\'{a}sz number that is an upper bound of the zero-error 
capacity of the corresponding quantum channel and which can be genuinely better than the one established in \cite{dsw}. 
Secondly, we continue the development of non-commutative graph theory by defining 
non-commutative counterparts of widely used classical graph parameters and establishing their interrelation.

In more detail, the paper is organised as follows. 
After some initial definitions and preliminary observations in Section \ref{s_dfo}, we
introduce in Section \ref{s_tst} \emph{non-commutative convex corners}, focusing on three convex corners 
associated with a non-commutative graph $\cl S \subseteq M_d$: 
the \emph{abelian projection corner} $\ap(\cl S)$, which we show to be a quantisation of the vertex packing polytope, 
and the anti-blockers $\cp(\cl S)^{\sharp}$ and $\fp(\cl S)^{\sharp}$ of the \emph{clique} and \emph{full projection corners}
$\cp(\cl S)$ and $\fp(\cl S)$, 
which turn out to be 
distinct quantisations of the fractional vertex packing polytope. We establish a first chain of inclusions between these 
convex corners, introduce several new non-commutative graph parameters that generalise the clique and the 
fractional clique numbers of a graph and of its complement, and evaluate these parameters in some special 
cases. 

In Section \ref{s_lc}, we introduce a non-commutative version $\thet(\cl S)$ of the theta-body of a graph
and establish the chain of inclusions
\begin{equation}\label{eq_apthe}
\ap(\cl S)\subseteq \thet(\cl S) \subseteq \fp(\cl S)^{\sharp}
\end{equation}
as a quantum version of (\ref{eq_clsacc}). 
Optimising the trace functional over (\ref{eq_apthe}) leads to a quantisation $\theta(\cl S)$ of the classical
Lov\'{a}sz number, different from the one 
introduced in \cite{dsw}, and to a numerical version of the inequalities (\ref{eq_clsa}).
We do not know whether $\theta$ is submultiplicative for the tensor product, and hence whether 
it is an upper bound of the zero-error capacity. 
This motivates the development in Section \ref{s_aqthe}, where we introduce 
yet another non-commutative version $\hat{\theta}(\cl S)$ of the Lov\'{a}sz number.
We show that $\hat{\theta}(\cl S)$ is an upper bound of the zero-error capacity of $\cl S$, 
which can be genuinely better than the non-commitative Lov\'{a}sz number of \cite{dsw}. 
In fact, we show that $\hat{\theta}(\cl S)$ is a genuine improvement of the 
complexity bound $\beta(\cl S)$ found in \cite{lpt}. 
Our results imply that the multiple characterisations \cite{lo} of the Lov\'{a}sz number of a graph lead to
(at least two) distinct parameters in the non-commutative case. 

In Section \ref{s_cp} we establish some further properties of the 
parameters introduced in the previous sections, the most important of which is 
the continuity of the maps $\cl S \to \thet(\cl S)$ and $\cl S \to \theta(\cl S)$.
While we do not know whether $\theta = \hat{\theta}$, we show that these two parameters 
take the minimal value $1$ only in the case of the complete non-commutative graph.
We prove the stability of the parameters $\theta$ and $\hat{\theta}$ under amplification, which 
constitutes another important difference between them and the parameter introduced in \cite{dsw}. 
We finish the paper with a short section containing some open problems.


\section{Definitions and basic properties}\label{s_dfo}

In this section, we set notation, recall some background from \cite{dsw}
and introduce various concepts that will be used in the sequel. 
Given a subset $S$ of a vector (resp. topological) space $V$, we denote by ${\rm conv}(S)$ 
(resp. $\overline{S}$) the convex hull (resp. the closure) of $S$.
We denote by $\bb{R}^d_+$ the set of all vectors in $\bb{R}^d$ with non-negative entries. 
Let $H$ be a Hilbert space of finite dimension $d$, which will be fixed throughout the paper 
unless stated otherwise. We denote by 
$\cl L(H)$ the algebra of all linear transformations on $H$, equipped with the operator norm $\|\cdot\|$. 
We denote by $I$ (or $I_d$) the identity operator on $H$. 
Given an orthonormal basis of $H$, we make the canonical identification $\cl L(H) \equiv M_d$.
We will often write $M_d$ in the place of $\cl L(H)$ even if we have not fixed a specific basis.
We denote by $\Tr$ the trace functional on $\cl L(H)$; if $A = (a_{i,j})_{i,j=1}^d\in \cl L(H)$ then
$\Tr(A) = \sum_{i=1}^d a_{i,i}$. 
We let $A^{\rm t}$ be the transpose of the matrix $A\in M_d$. 
We use $\langle \cdot,\cdot\rangle$ to denote both vector space duality and inner products, which we assume to be 
linear on the first variable. 
Note that the dual space of $M_d$ can be canonically identified with $M_d$ via the pairing 
$\langle A,B\rangle = \Tr(AB)$. 
We equip $M_d$ with the Hilbert-Schmidt inner product $(A,B)\to \Tr(B^*A)$, $A,B\in M_d$. 
As usual, given a subspace $\cl F$ of a Hilbert space, $\cl F^{\perp}$ denotes its orthogonal complement.
If $\xi, \eta\in H$, we write $\xi\eta^*$ for the rank one operator on $H$ given by $(\xi\eta^*)(\zeta) = \langle \zeta,\eta\rangle\xi$.

A subspace $\cl S\subseteq \cl L(H)$ is called an \emph{operator system} if $I\in \cl S$ and 
$A^*\in \cl S$ whenever $A\in \cl S$. In this case, we say that $\cl S$ is a \emph{non-commutative graph on} $H$.
We denote by $\cl S^+$ the cone of all positive operators in $\cl S$. 
It is clear that if $\cl S\subseteq \cl L(H)$ is an operator system and $m\in \bb{N}$ 
then the space $M_m(\cl S)$ of all $m$ by $m$ matrices with entries in $\cl S$ is an operator system 
in $\cl L(H^m)$, where $H^m$ is the direct sum of $m$ copies of $H$. 

Let $G = (X,E)$ be an undirected graph without loops, with vertex set $X$ of cardinality $d$ and edge set $E$. 
We denote by $G^c$ the graph complement of $G$, that is, $G^c = (X,\tilde{E})$ where,
for $x\neq y$, we have that $\{x,y\}\in \tilde{E}$ if and only if $\{x,y\}\not\in E$.
We  write $x\sim y$ if $\{x,y\}\in E$ and $x\simeq y$ if $x\sim y$ or $x = y$. 
Identifying $X$ with $[d] := \{1,\dots,d\}$, we let $(e_x)_{x\in X}$ be the canonical orthonormal basis of 
$H \cong \bb{C}^d$, and set 
$$\cl S_G = {\rm span}\{e_xe_{y}^* :  x,y\in X, x\simeq y\}.$$
Let $\cl D_X$ be the diagonal matrix algebra 
corresponding to the basis $(e_x)_{x\in X}$ and 
$\Delta : M_d\to \cl D_X$ be the conditional expectation.
(We sometimes write $\cl D_d$ in the place of $\cl D_X$.) 
For a subset $F\subseteq X$, we let $\chi_F$ be the characteristic function of $F$ 
and set $P_F = \sum_{x\in F} e_x e_x^*$. 
We have that $\cl S_G$ is an operator system and a $\cl D_X$-bimodule in the sense that 
$BTA\in \cl S_G$ whenever $T\in \cl S_G$ and $A,B\in \cl D_X$. 
Operator systems of the form $\cl S_G$ for some graph $G$ will be called 
\emph{graph operator systems}; it is straightforward to see that these are precisely the 
operator systems acting on $H$ that are $\cl D_X$-bimodules. 

Let $\cl S$ and $\cl T$ be operator systems. A linear map $\nph : \cl S\to \cl T$ 
is called unital if $\nph(I) = I$, and completely positive if 
$\nph^{(m)}(M_m(\cl S)^+)\subseteq M_m(\cl T)^+$ for every $m\in \bb{N}$, 
where $\nph^{(m)} : M_m(\cl S)\to M_m(\cl T)$ is the map given by $\nph^{(m)}((a_{i,j})) = (\nph(a_{i,j}))$. 
The map $\nph$ is called a complete order isomorphism if $\nph$ is completely positive, 
bijective and $\nph^{-1}$ is completely positive. 
It was shown in \cite{op} that, if $G_1$ and $G_2$ are graphs then 
$\cl S_{G_1}$ is unitally 
completely order isomorphic to $\cl S_{G_2}$ precisely when $G_1$ is graph isomorphic to $G_2$. 

The concept in (i) of the following definition was introduced in \cite{dsw}.

\begin{definition}\label{d_inde}
Let $H$ be a finite dimensional Hilbert space and $\cl S\subseteq \cl L(H)$ be an operator system. 
A set $\{\xi_i\}_{i=1}^m\subseteq H$ of mutually orthogonal unit vectors is called 

(i) \ \emph{$\cl S$-independent} 
if $\{\xi_i\xi_j^* : i\neq j\}\subseteq \cl S^{\perp}$;

(ii)
\emph{$\cl S$-clique} 
if $\{\xi_i\xi_j^* : i\neq j\}\subseteq \cl S.$
\end{definition}


\begin{definition}\label{d_ab}
Let $H$ be a finite dimensional Hilbert space and 
$\cl S\subseteq \cl L(H)$ be an operator system. A projection $P\in \cl L(H)$ will be called

(i) \ \ \emph{$\cl S$-abelian} 
if $P\cl S P$ is contained in an abelian C*-subalgebra of $\cl L(H)$;

(ii) \ \emph{$\cl S$-full}  if $\cl L(PH)\oplus 0_{P^{\perp}} \subseteq \cl S$.

(iii) \emph{$\cl S$-clique} if its range is the span of an $\cl S$-clique. 
\end{definition}

We denote the set of all $\cl S$-abelian (resp. $\cl S$-full, $\cl S$-clique) projections by 
$\cl P_{\rm a}(\cl S)$ (resp. $\cl P_{\rm f}(\cl S)$, $\cl P_{\rm c}(\cl S)$). 

\medskip

\noindent{\bf Remarks. (i) } 
The condition $\cl L(PH)\oplus 0_{P^{\perp}} \subseteq \cl S$ will often be written 
simply $\cl L(PH)\subseteq \cl S$.
If a projection $P$ is $\cl S$-full then $P\in \cl S$.

\smallskip

{\bf (ii) } 
Every $\cl S$-full projection is $\cl S$-clique. The converse does not hold true
even in the case where $\cl S$ is a graph operator system. 
For example, let $G$ be the full bipartite graph between sets $X$ and $Y$
(so that $V(G) = X\cup Y$, with $X$ and $Y$ disjoint), where $|X| > 1$.
Let $v = \frac{1}{\sqrt{|X|}}\chi_{X}$ and $w = \frac{1}{\sqrt{|Y|}}\chi_{Y}$, viewed as (unit) vectors in $\bb{C}^{|V(G)|}$. 
Then $\{v,w\}$ is an $\cl S_G$-clique, but the projection onto ${\rm span}\{v,w\}$ is not $\cl S$-full since 
no two vertices in $X$ are adjacent. 

\medskip


Part (i) of the next proposition was communicated to us by Vern I. Paulsen.

\begin{proposition}\label{p_int}
Let $H$ be a finite dimensional Hilbert space and $\cl S\subseteq \cl L(H)$ be an operator system. 

(i) \ 
A projection $P\in \cl L(H)$ is $\cl S$-abelian if and only if 
there exists an orthonormal basis of $PH$ that is an $\cl S$-independent set.

(ii) A projection $P\in \cl S$ is $\cl S$-full if and only if every orthonormal basis of $PH$ is an $\cl S$-clique.
\end{proposition}

\begin{proof}
(i) 
Suppose that the orthonormal set $\{\xi_i\}_{i=1}^m \subseteq H$ 
is $\cl S$-independent and let $P$ be the projection onto its span. 
If $T\in \cl S$ and $i\neq j$ then $T\perp \xi_i\xi_j^*$ and thus
\begin{equation}\label{eq_PTP}
PTP 
= 
\sum_{i,j=1}^m (\xi_i\xi_i^*) T (\xi_j\xi_j^*)
= \sum_{i,j=1}^m \langle T\xi_j, \xi_i\rangle (\xi_i\xi_j^*)
= 
\sum_{i=1}^m \langle T\xi_i, \xi_i\rangle (\xi_i\xi_i^*); 
\end{equation}
hence, $PTP$ is contained in the abelian algebra ${\rm span}\{\xi_i\xi_i^* : i = 1,\dots,m\}$. 

Conversely, assume that $P$ is $\cl S$-abelian, and let $\cl D\subseteq \cl L(PH)$ be a maximal
abelian C*-subalgebra such that $P\cl S P\subseteq \cl D$. Let $(\xi_i)_{i=1}^m$  be a 
family of mutually orthogonal unit vectors whose span is $P$ 
such that ${\rm span}\{\xi_i\xi_i^* : i = 1,\dots,m\} \subseteq \cl D$. 
Since $(\xi_i\xi_j^*)_{i,j}$ is a linearly independent family, (\ref{eq_PTP}) shows that 
$\langle T\xi_j, \xi_i \rangle = 0$ whenever $i\neq j$; thus, 
the set $\{\xi_i\}_{i = 1}^m$ is $\cl S$-independent. 

(ii) 
Suppose that every orthonormal basis of $PH$ is an $\cl S$-clique.
Fix an $\cl S$-clique $\{\xi_i\}_{i = 1}^k$ that spans $PH$,
and $i,j$ with $1\leq i\neq j\leq k$.
Then $\xi_i\xi_j^*\in \cl S$. 
Since $\xi_i + \xi_j \perp \xi_i - \xi_j$, we have 
$$\xi_i \xi_i^* + \xi_j\xi_i^* - \xi_i\xi_j^* - \xi_j\xi_j^* = (\xi_i + \xi_j)(\xi_i - \xi_j)^* \in \cl S.$$
Thus, $\xi_i \xi_i^* - \xi_j\xi_j^* \in \cl S$.
It follows that 
$$k (\xi_1\xi_1^*) = P + \sum_{i = 2}^k (\xi_1\xi_1^* - \xi_i\xi_i^*) \in \cl S$$
and so $\xi_1\xi_1^*\in \cl S$. 
By symmetry, $\xi_i\xi_i^*\in \cl S$ for all $i \in [k]$. 
Since the family $\{\xi_i\}_{i=1}^k$ is an $\cl S$-clique,
we conclude that $\cl L(PH) = {\rm span}\{\xi_i \xi_j^* : i,j \in [k]\}\subseteq \cl S$,
and so $P$ is $\cl S$-full. 

Conversely, suppose that $P$ is $\cl S$-full, that is, $\cl L(PH) \subseteq \cl S$. 
If $\{\xi_i\}_{i=1}^k$ is an orthonormal basis of $PH$ then 
clearly $\xi_i\xi_j^*\in \cl L(PH)$ and hence $\xi_i\xi_j^*\in \cl S$, for all $i\neq j$. 
Thus, $\{\xi_i\}_{i = 1}^k$  is an $\cl S$-clique. 
\end{proof}

\medskip

We next consider a natural candidate for a graph complement in the non-commutative case.

\begin{definition}\label{d_compl}
Let $\cl S$ be a non-commutative graph. The \emph{complement} $\cl S^c$ of $\cl S$ is 
the operator system $\cl S^c = \cl S^{\perp} + \bb{C}I$. 
\end{definition}

\begin{remark}\label{r_cc}
If $\cl S$ is a non-commutative graph then $\cl S = \cl S^{cc}$. 
\end{remark}

\begin{proof}
Clearly, 
\begin{equation}\label{eq_cont}
\cl S^{cc} = (\cl S^{\perp} + \bb{C}I)^{\perp} + \bb{C}I\subseteq \cl S.
\end{equation}
The equality follows from the fact that the left and the right hand side in 
(\ref{eq_cont}) have the same dimension.
\end{proof}

\begin{proposition}\label{p_coic}
Let $H$ be a finite dimensional Hilbert space and $\cl S\subseteq \cl L(H)$ be an operator system. 

(i) \ A subset $\{\xi_i\}_{i=1}^k \subseteq H$ is $\cl S$-independent if and only if it is an $\cl S^c$-clique. 

(ii) A subset $\{\xi_i\}_{i=1}^k \subseteq H$ is an $\cl S$-clique if and only if it is $\cl S^c$-independent. 

\noindent Thus, a projection $P$ is $\cl S$-abelian if and only if $P$ is $\cl S^c$-clique.
\end{proposition}

\begin{proof}
Suppose that the set $\{\xi_i\}_{i=1}^k \subseteq H$ is $\cl S$-independent; thus,
$\{\xi_i\xi_j^* : i\neq j\}\subseteq \cl S^{\perp}$. Since 
$\cl S^{\perp}\subseteq \cl S^c$, we conclude that 
$\{\xi_i\}_{i=1}^k$ is an $\cl S^c$-clique.

Suppose that $\{\xi_i\}_{i=1}^k$ is an $\cl S$-clique. 
Then $\xi_i\xi_j^* \perp \cl S^{\perp}$ whenever $i\neq j$. 
Trivially, $\xi_i\xi_j^*\perp \bb{C}I$ whenever $i\neq j$; thus,
$\xi_i\xi_j^*\perp (\cl S^{\perp} + \bb{C}I)$ whenever $i\neq j$, and so 
the set $\{\xi_i\}_{i=1}^k$ is $\cl S^c$-independent. 

(i) Suppose that $\{\xi_i\}_{i=1}^k$ is an $\cl S^c$-clique. By the previous paragraph, 
$\{\xi_i\}_{i=1}^k$ is $\cl S^{cc}$-independent. By Remark \ref{r_cc}, 
$\{\xi_i\}_{i=1}^k$ is $\cl S$-independent.

(ii) Suppose that $\{\xi_i\}_{i=1}^k$ is an $\cl S^c$-independent set. 
By the first paragraph, $\{\xi_i\}_{i=1}^k$ is an $\cl S^{cc}$-clique.
By Remark \ref{r_cc}, $\{\xi_i\}_{i=1}^k$ is an $\cl S$-clique.

The remaining claims follow from Proposition \ref{p_int}.
\end{proof}

\begin{remark}\label{r_closeds}
{\rm 
Let $H$ be a finite dimensional Hilbert space and $\cl S\subseteq \cl L(H)$ be an operator system.
The sets $\cl P_{\rm a}(\cl S)$, $\cl P_{\rm f}(\cl S)$ and $\cl P_{\rm c}(\cl S)$ are closed.
}
\end{remark}

\begin{proof}
Suppose that $(P_n)_{n\in \bb{N}}$ is a sequence of $\cl S$-full projections with $\lim_{n\to\infty}$ $P_n = P$. 
For every $A\in \cl L(H)$ with $A = PAP$ we have $A = \lim_{n\to\infty} P_nAP_n$; 
since $P_n A P_n\in \cl S$ for each $n$ and $\cl S$ is closed, we have that $A\in \cl S$. 
Thus, $\cl P_{\rm f}(\cl S)$ is closed. 

Assume that 
$(P_n)_{n\in \bb{N}}$ is a convergent sequence of $\cl S$-abelian projections with limit $P$. 
For all $A,B\in \cl S$, we have 
\begin{eqnarray*}
(PAP)(PBP) & = & \lim_{n\to \infty} (P_nAP_n)(P_nBP_n) = \lim_{n\to \infty} (P_nBP_n)(P_nAP_n)\\ 
& = & (PBP)(PAP).
\end{eqnarray*}
It follows that $\cl P_{\rm a}(\cl S)$ is closed; by Proposition \ref{p_coic}, $\cl P_{\rm c}(\cl S)$ is closed, 
and the proof is complete.
\end{proof}


\section{The first sandwich theorem}\label{s_tst}

In this section, we prove the first of our sandwich theorems. For clarity, the section is split in three subsections. 

\subsection{Convex corners from non-commutative graphs}\label{ss_ccng} 

Part (ii) of the following definition contains a classical notion arising in Graph Theory \cite{cklms}, 
while part (i) introduces a suitable non-commutative version that will play a central role subsequently.

\begin{definition}\label{d_ccinmd}
(i) \ Let $H$ be a Hilbert space of dimension $d$.
A \emph{convex corner} in $\cl L(H)$ (or in $M_d$)
is a non-empty closed convex subset $\cl A$ of $\cl L(H)^+$ such that 
\begin{equation}\label{eq_cc}
A\in \cl A \mbox{ and } 0\leq B\leq A  \mbox{ imply } B\in \cl A;
\end{equation}

(ii) If $d\in \bb{N}$,
a \emph{diagonal convex corner} in $M_d$ is 
a non-empty closed convex subset $\cl C$ of $\cl D_d^+$, such that 
\begin{equation}\label{eq_ccc}
A\in \cl A, \ B\in \cl D_d \mbox{ and } 0\leq B\leq A  \mbox{ imply } B\in \cl A.
\end{equation}
\end{definition}

Conditions (\ref{eq_cc}) and (\ref{eq_ccc}) will be referred to as \emph{hereditarity}. 
If $\cl A$ is a non-empty subset of $\cl L(H)^+$, let
$$\cl A^{\sharp} = \left\{B\in \cl L(H)^+ : \Tr(AB)\leq 1, \mbox{ for all } A\in \cl A\right\}$$
and call $\cl A^{\sharp}$ the \emph{anti-blocker} of $\cl A$. 
Similarly \cite{cklms}, if $\cl C$ is a non-empty subset of $\cl D_d^+$, 
let
$$\cl C^{\flat} = \left\{B\in \cl D_d^+ : \Tr(AB)\leq 1 \mbox{ for all } A\in \cl C\right\}$$
and call $\cl C^{\flat}$ the \emph{diagonal anti-blocker} of $\cl C$.
The following facts are immediate.

\begin{remark}\label{r_trcc}
{\rm 
Let $\cl A\subseteq M_d^+$ (resp. $\cl C\subseteq \cl D_d^+$) be a non-empty set. Then 

(i) \ the set $\cl A^{\sharp}$ (resp. $\cl C^{\flat}$) is a convex corner (resp. a diagonal convex corner) in $M_d$;

(ii) $\cl A^{\sharp\sharp\sharp} = \cl A^{\sharp}$ and $\cl C^{\flat\flat\flat} = \cl C^{\flat}$.
}
\end{remark}

Let $H$ be a finite dimensional Hilbert space. 
For a subset $\cl C\subseteq \cl L(H)$, we let 
$$\her(\cl C) = \left\{A\in \cl L(H)^+ : \ \exists \ B\in \cl C \mbox{ such that } A\leq B\right\}.$$

\begin{proposition}\label{l_gen}
Let $\cl P\subseteq \cl L(H)^+$ be a non-empty bounded set and $\cl A = \her(\overline{\conv}(\cl P))$. 

(i) \ The set $\cl A$ is a convex corner. Moreover, $\cl A^{\sharp} = \cl P^{\sharp}$. 

(ii) Assume that $\cl P$ is a closed set of projections such that if $P\in \cl P$ and $P'$ is a projection with 
$P'\leq P$ then $P'\in \cl P$. If $Q$ is a projection with $Q\in \cl A$ then $Q\in \cl P$. 
\end{proposition}

\begin{proof}
(i) It is clear that $\cl A$ is hereditary. 
Since $\overline{\rm conv}(\cl P)$ is convex, $\cl A$ is convex. Suppose that $(T_n)_{n\in \bb{N}}\subseteq \cl A$
and $T_n\to_{n\to\infty} T$. Let $C_n\in \overline{\rm conv}(\cl P)$ be such that $T_n\leq C_n$, $n\in \bb{N}$. 
Since $\overline{\rm conv}(\cl P)$ is compact, $(C_n)_{n\in \bb{N}}$ has a cluster point, say $C$, in $\overline{\rm conv}(\cl P)$. 
But then $T\leq C$ and hence $\cl A$ is closed. 

Since $\cl P\subseteq \cl A$, we have that $\cl A^{\sharp} \subseteq \cl P^{\sharp}$.
The reverse inclusion follows from basic properties of the trace functional. 

(ii) Let $T\in \overline{{\rm conv}}(\cl P)$ be such that $Q\leq T$. Then $Q\leq QTQ$ and hence 
$1 = \|Q\| \leq \|QTQ\| \leq \|T\| \leq 1$, showing that $\|QTQ\| = 1$. Thus, $QTQ \leq Q$ and hence 
$Q = QTQ \in \overline{{\rm conv}}(Q\cl P Q)$. Since $Q$ is an extreme point of 
the unit ball of $\cl L(H)^+$, we have that $Q\in \overline{Q\cl P Q} = Q\overline{\cl P}Q = Q\cl P Q$. 
Let $P\in \cl P$ be such that $Q = QPQ$. We have that $P = Q + P'$ for some projection $P'\leq Q^{\perp}$. 
In particular, $Q\leq P$ and since the set $\cl P$ is hereditary, we conclude that $Q\in \cl P$. 
\end{proof}

Let $\cl S\subseteq \cl L(H)$ be an operator system. Set 
\begin{itemize}
\item $\ap(\cl S) = \her\left(\overline{\conv}\left\{P : P \mbox{ an $\cl S$-abelian projection}\right\}\right)$;
\item $\cp(\cl S) = \her\left(\overline{\conv}\left\{P : P \mbox{ an $\cl S$-clique projection}\right\}\right)$;
\item $\fp(\cl S) = \her\left(\overline{\conv}\left\{P : P \mbox{ an $\cl S$-full projection}\right\}\right)$.
\end{itemize}
We call $\ap(\cl S)$ (resp. $\cp(\cl S)$, $\fp(\cl S)$) the \emph{abelian} 
(resp. \emph{clique}, \emph{full}) \emph{projection convex corner} of $\cl S$. By 
Proposition \ref{l_gen}, these are indeed convex corners while, by 
Proposition \ref{p_coic}, $\ap(\cl S) = \cp(\cl S^c)$.

\begin{remark}\label{r_cpfpi}
{\rm (i) 
For any non-commutative graph $\cl S \subseteq \cl L(H)$, every rank one projection on $H$ 
is $\cl S$-abelian and $\cl S$-clique. Thus, 
$\ap(\cl S)$ and $\cp(\cl S)$ always contain the convex corner $\{T\in \cl L(H)^+ : \Tr(T) \leq 1\}$.
On the other hand, $\fp(\cl S)$ may be zero, e.g. in the case 
where $\cl S = $ ${\rm span}\{I, E_{1,2}, E_{1,3},$ $E_{2,1}, E_{3,1}\}\subseteq M_3$.

(ii) 
By Remark (ii) after Definition \ref{d_ab}, $\fp(\cl S)\subseteq \cp(\cl S)$. 
Strict inclusion may occur even in the case where $\fp(\cl S)\neq \{0\}$, for example, 
if $\cl S = {\rm span}\{E_{1,2}, E_{2,1}, I_2\}\subseteq M_2$. 
}
\end{remark}

Let $G = (X,E)$ be a graph on $d$ vertices. Recall that a subset $S\subseteq X$ 
is called \emph{independent} (resp. a \emph{clique}) if whenever $x,y\in S$ and $x\neq y$, we have that $x\not\sim y$ 
(resp. $x\sim y$). 
The \emph{vertex packing polytope} \cite{gls} of $G$ is the set
$$\vp(G) = \conv\left\{\chi_{S} : S\subseteq X \mbox{ an independent set}\right\},$$
while the \emph{fractional vertex packing polytope} \cite{gls} of $G$ is the set
$$\fvp(G) = \left\{x\in \bb{R}_+^d : \sum_{i\in K}x_i \leq 1,  \mbox{ for all cliques } K\subseteq X\right\}.$$
These sets are diagonal convex corners in $M_d$, if 
we identify an arbitrary element $v = (v_i)_{i=1}^d$ of $\bb{R}^d$ with the matrix 
with entries $v_1,\dots,v_d$ down the diagonal and zeros elsewhere (see \cite{knuth}).

We next show that $\ap(\cl S)$ is a suitable non-commutative version of $\vp(G)$,
while $\cp(\cl S)^{\sharp}$ and $\fp(\cl S)^{\sharp}$ are suitable non-commutative versions of $\fvp(G)$.


\begin{theorem}  \label{apvp} 
Let $G = (X,E)$ be a graph on $d$ vertices. Then 

(i) \ \ $\Delta ( \ap (\cl S_G)) = \cl D_X \cap  \ap (\cl S_G) = \vp (G)$;

(ii) \ $\Delta ( \cp (\cl S_G)) = \cl D_X \cap  \cp (\cl S_G) = \vp (G^c)$;  
  
(iii) $\Delta ( \fp (\cl S_G)) = \cl D_X \cap \fp (\cl S_G) = \vp (G^c)$.
\end{theorem}

\begin{proof} 
(i) Let $S$ be an independent set in $G$.  
Then $e_{x}e_{y}^* \in \cl S_G^\perp$ for $x,y\in S$ with $x\neq y$, and 
hence $\sum_{x\in S} e_x e_x^*$ is an $\cl S_G$-abelian projection in $\cl D_X$. 
Since $\cl D_X \cap ( \ap (\cl S_G))$ is a convex set, this implies  
$\vp(G) \subseteq \cl D_X \cap \ap(\cl S_G) \subseteq \Delta(\ap(\cl S_G)).$  

It remains to show that $\Delta(\ap(\cl S_G)) \subseteq \vp(G)$. 
Let $\{v_1, \ldots, v_m \}$ be an $\cl S_G$-independent set; clearly, $m \le d$. 
Suppose that $m = d$; then $I\in \ap(\cl S_G)$ and hence $\ap(\cl S_G) = \{T\in M_d^+ : \|T\| = 1\}$.
Note that
$\cl S_G\subseteq {\rm span}\{v_iv_i^* : i\in [m]\}$.
In particular, $e_xe_x^*\in {\rm span}\{v_iv_i^* : i\in [m]\}$ for every $x\in X$, 
and this easily implies that $\{v_i\}_{i=1}^m = \{\zeta_xe_x\}_{x\in X}$, for some 
unimodular constants $\zeta_x\in \bb{C}$, $x\in X$. It follows that $G$ is the empty graph and hence 
$\vp(G) = \{A\in \cl D_d^+ : \|A\| \leq 1\}$.
It is now clear that $\Delta(\ap(\cl S_G)) \subseteq \vp(G)$. 

Suppose that $m < d$ and set $P = \sum_{i=1}^m v_iv_i^*$. 
Write 
$v_i = \sum_{x\in X} \lambda_x^{(i)} e_x$ and $a^{(i)}_x= \left|\lambda_x^{(i)}\right|^2$, $i\in [m]$, $x\in X$.
 Let $i \ne j$; then 
$v_i v_j^*= \sum_{x,y} \lambda_x^{(i)} \overline{\lambda}_y^{(j)} e_xe_y^* \in \cl S_G^\perp.$  
Thus,
\begin{equation} \label{SGind} 
a_x^{(i)} a_y^{(j)} \ne 0 \ \Longrightarrow \ x \not \simeq y \mbox{ in } G.
\end{equation} 

Now 
$$P = \sum_{i=1}^m \sum_{x,y \in X} \lambda_x^{(i)}\overline{\lambda}_y^{(i)} e_xe_y^*,$$ 
and so 
\begin{equation} \label{deltaP} 
\Delta (P) = \sum_{x\in X} \sum_{i=1}^{m} a_x^{(i)} e_xe_x^*. 
\end{equation}  
Note that $\sum_{x\in X} a_x^{(i)} = \ip{v_i}{v_i} = 1$, $i= 1,\ldots, m$. 
Since $\Delta(P) \le I$, we have $\sum_{i=1}^m a_x^{(i)} \le 1$, $x\in X$.
For each $x\in X$, let 
$r_x = 1-\sum_{i=1}^m a_x^{(i)}$; thus, $r_x \geq 0$. 
Set 
\begin{align*} 
m_{i,x} =  &  \ a_x^{(i)}, \mbox{ if } 1 \le i \le m, \\ 
m_{i,x} =  & \ \frac{r_x}{d-m},  \mbox{ if } m+1 \le i \le d,
\end{align*}
and let $M = \left( m_{i,x}\right)_{i,x}$; thus, $M\in M_d$.
Observe that the matrix $M$ is doubly stochastic. Indeed,
if $1\leq i \leq m$ then $\sum_{x\in X} m_{i,x} = \sum_{x\in X} a_x^{(i)} = 1$.  
On the other hand, 
$d - \sum_{x \in X} r_x = \sum_{i=1}^m \sum_{x\in X} a_x^{(i)} = m$ and hence, if
$m+1 \le i \le d$ we have $\sum_{x\in X} m_{i,x} = \sum_{x\in X} \frac{r_x}{d-m} = 1$. 
Finally, if $x\in X$ then $\sum_{i=1}^d m_{i,x} = \sum_{i=1}^d a_x^{(i)} + r_x = 1$.

By the Birkhoff-von Neumann theorem, 
there exist $l\in \bb{N}$, $\gamma_k > 0$
and permutation matrices $P^{(k)} = \left( p^{(k)}_{i,x} \right)_{i,x} \in M_d$, $k \in [l]$, 
such that $\sum_{k=1}^l \gamma_k =1$ and 
\begin{equation}\label{eq_Mgamma}
M = \sum_{k=1}^l \gamma_k P^{(k)}.
\end{equation}

Set
$$Q_k = \sum_{x\in X} \sum_{i=1}^{m} p_{i,x}^{(k)} e_xe_x^*.$$  
For $x\in X$ and $1 \le i \le m$, by (\ref{eq_Mgamma}) we have 
\begin{equation}\label{eq_axg}
a_x^{(i)} = \sum_{k=1}^l \gamma_k p_{i,x}^{(k)}
\end{equation}
and so 
\eqref{deltaP} implies
\begin{equation}\label{eq_deltap}
\Delta (P)=\sum_{k=1}^l  \gamma_k Q_k.
\end{equation}

Suppose that $i,j\in [m]$ and $x,y\in X$ are such that 
$(i,x) \neq (j,y)$ and $p_{i,x}^{(k)} = p_{j,y}^{(k)} = 1$.  
Since $P^{(k)}$ is a  permutation matrix, $i \neq j$ and $x \neq y$.  
By (\ref{eq_axg}),
$a_x^{(i)} \ne 0$ and $a_y^{(j)} \ne 0$; by \eqref{SGind}, $x$ and $y$ are (distinct and) non-adjacent vertices in $G$.  
Thus, each $Q_k$ is a projection in $\vp(G)$ and (\ref{eq_deltap}) implies that 
$\Delta(P) \in \vp(G)$.
It follows that $\Delta(T) \in \vp(G)$ whenever $T\in \conv\{P : P \mbox{ is } \cl S_G\mbox{-abelian}\}$. 
Since $\vp(G)$ is closed, $\Delta(T) \in \vp(G)$ whenever 
$T\in \overline{\conv}\{P : P \mbox{ is } \cl S_G\mbox{-abelian}\}$; 
since $\vp(G)$ is hereditary, $\Delta(T) \in \vp(G)$ whenever $T\in \ap(\cl S_G)$,
and the proof of (i) is complete.

\smallskip

(ii)-(iii) 
Let $K$ be an independent set in $G^c$, that is, a clique in $G$.  
Then $e_{x}e_{y}^* \in \cl S_G$ for all $x,y\in K$ and so  
$\sum_{x\in K}e_{x}e_{x}^*$ is an $\cl S_G$-full  projection. 
Together with Remark \ref{r_cpfpi}, this implies
$$\vp(G^c) \subseteq \cl D_X \cap \fp(\cl S_G) \subseteq \cl D_X \cap \cp(\cl S_G) \subseteq \Delta(\cp(\cl S_G))$$  
and 
$$\vp(G^c) \subseteq \Delta( \fp(\cl S_G)) \subseteq \Delta(\cp(\cl S_G)).$$   

It remains to show that $\Delta(\cp(\cl S_G)) \subseteq \vp(G^c)$. 
Let $\{v_1, \ldots, v_m \}$ be an $\cl S_G$-clique.
Suppose that $m = d$. 
Since $\cl S_G^{\perp}\subseteq {\rm span}\{v_i v_i^* : i\in [m]\}$, we have that 
$\cl S_G^{\perp}$ is a commutative family of operators. 
This easily implies that $G$ is the complete graph; thus, 
$\cp(\cl S_G) = \{T\in M_d : 0 \leq T \leq I\}$, 
$\vp(G^c) = \{T\in \cl D_d : 0 \leq T \leq I\}$, 
and $\Delta(\cp(\cl S_G)) = \vp(G^c)$. 

Assume that $m < d$ and let
$P = \sum_{i=1}^m v_iv_i^*$ be the corresponding $\cl S_G$-clique projection.  
Writing $v_i = \sum_{x\in X} \lambda_x^{(i)}e_x$, 
similarly to the proof of (i), we see that 
if $i \ne j$ and $\lambda_x^{(i)} \lambda_y^{(j)} \ne 0$ then 
$x \simeq y$ in $G$. 
Following the proof of (i) we now obtain
that $\Delta(P) \in \vp(G^c)$, and consequently that
$\Delta(\cp(\cl S_G)) \subseteq \vp(G^c)$, completing the proof.  
\end{proof}

\begin{lemma} \label{Deltasharp} 
Let $\cl A$ be a diagonal convex corner, and $\cl B$ be a convex corner, in $M_d$, such that 
\begin{equation} \label{AB} 
\cl A= \cl D_d \cap \cl B =  \Delta (\cl B).
\end{equation} 
Then 
$\cl A^\flat=\cl D_d \cap  \cl B^\sharp = \Delta ( \cl B^\sharp).$
\end{lemma}

\begin{proof} 
If \eqref{AB} holds then $\cl A \subseteq  \cl B$ and so $ \cl B^\sharp \subseteq  \cl A^\sharp$.  
Thus, $\cl D_d \cap  \cl B^\sharp \subseteq \cl D_d \cap \cl A^\sharp = \cl A^\flat$.   
For the reverse inclusion, let $T \in \cl A^\flat$ and $N \in  \cl B$.  
By  \eqref{AB}, $\Delta(N) \in \cl A.$  
Hence 
$$\Tr\left(TN\right) = \Tr\left(\Delta(T)N\right) = \Tr\left(T\Delta(N)\right) \le 1,$$ 
and so
$T \in  \cl B^\sharp \cap \cl D_d$.  
Thus, $\cl A^\flat \subseteq  \cl D_d \cap  \cl B^\sharp$. 

Trivially, $\cl D_d \cap  \cl B^\sharp \subseteq \Delta ( \cl B^\sharp).$  
Fix $M \in \Delta (  \cl B^\sharp)$ and $N \in  \cl B$, and let $R \in  \cl B^\sharp$ be such that 
$M = \Delta(R)$.  
By   \eqref{AB}, $\Delta(N) \in  \cl B$ and hence 
$$\Tr (MN) = \Tr (\Delta(R)N) = \Tr (R \Delta(N)) \le 1,$$
giving $M \in  \cl B^\sharp \cap \cl D_d.$  
Thus $ \Delta(\cl B^\sharp) \subseteq \cl D_d \cap  \cl B^\sharp$ and the proof is complete.
\end{proof}

\begin{corollary} \label{antiblockvp} 
Let $G = (X,E)$ be a graph on $d$ vertices. Then 

(i) \ \ $\Delta \left(\ap (\cl S_G)^\sharp\right) = \cl D_X \cap  \ap \left(\cl S_G\right)^\sharp = \vp(G)^\flat$;

(ii) \ $\Delta \left( \cp (\cl S_G)^\sharp\right) = \cl D_X \cap \cp \left(\cl S_G\right)^\sharp = \fvp(G)$;   
  
(iii) $\Delta \left(\fp (\cl S_G)^\sharp\right) = \cl D_X \cap \fp \left(\cl S_G\right)^\sharp = \fvp(G)$.
\end{corollary}

\begin{proof}
Immediate from  Lemma \ref{Deltasharp}, Theorem \ref{apvp} and the fact that $\fvp(G) = \vp(G^c)^{\flat}$.
\end{proof}


\subsection{Non-commutative graph parameters}\label{ss_ncgp}

In this subsection, we introduce various parameters of non-commutative graphs 
and point out their relation with classical graph parameters. 

If $\cl A \subseteq M_d$ is a bounded set, let 
\begin{equation}\label{eq_thtr}
\theta(\cl A) = \sup\left\{\Tr(A) : A\in \cl A\right\}.
\end{equation}

\begin{remark}\label{r_thetaher}
{\rm 
If $\cl P \subseteq M_d$ is a bounded set and 
$\cl A = {\rm her}\left(\overline{{\rm conv}}(\cl P)\right)$ then 
$\theta(\cl A) = \theta(\cl P)$. 
%
}
\end{remark}

\begin{proof}
It is clear that $\theta(\cl A) = \Tr(A)$ for some $A\in \overline{{\rm conv}}(\cl P)$. 
Since the trace is affine and continuous, 
by Bauer's Maximum Principle (see \cite[7.69]{ab}), $A$ can be chosen to be an extreme point 
of $\overline{{\rm conv}}(\cl P)$.
By Milman's Theorem, $A\in \overline{\cl P}$; thus $\theta(\cl A) = \theta(\cl P)$.
\end{proof}

Let $H$ be a $d$-dimensional Hilbert space and $\cl S\subseteq \cl L(H)$ be an operator system. 
We set 
\begin{itemize}
\item[(i)] $\alpha(\cl S) = \theta\left(\ap(\cl S)\right)$ -- the \emph{independence number} of $\cl S$ \cite{dsw};
\item[(ii)] $\omega(\cl S) = \theta\left(\cp(\cl S)\right)$ -- the \emph{clique number} of $\cl S$;
\item[(iii)] $\tilde{\omega}(\cl S) = \theta\left(\fp(\cl S)\right)$ -- the \emph{full number} of $\cl S$;
\item[(iv)] $\omega_{\rm f}(\cl S) = \theta\left(\ap(\cl S)^{\sharp}\right)$ -- the 
\emph{fractional clique number} of $\cl S$;
\item[(v)] $\kappa(\cl S) = \theta\left(\cp(\cl S)^{\sharp}\right)$ -- the 
\emph{complementary fractional clique number} of $\cl S$;
\item[(vi)] $\nph(\cl S) = \theta\left(\fp(\cl S)^{\sharp}\right)$ -- the 
\emph{complementary fractional full number} of $\cl S$;
\item[(vii)] $\chi(\cl S) = \min\{k\in \bb{N} : \mbox{ there exist } \ \cl S\mbox{-abelian projections } P_1,\dots,P_k\\ \mbox{ with } 
\sum_{l=1}^k P_l = I\}$
-- the \emph{chromatic number} of $\cl S$ \cite{p_lect}.
\end{itemize}


It follows from Remark \ref{r_thetaher} that the parameters $\alpha$, $\omega$ and $\tilde{\omega}$
take non-negative integer values. In fact, by Remark \ref{r_closeds}, 
$\alpha(\cl S)$ (resp. $\omega(\cl S)$) coincides with the maximum size of an 
$\cl S$-independent set (resp. an $\cl S$-clique).

\smallskip

A subspace $\cl J\subseteq \cl L(H)$ will be called an \emph{operator anti-system} if 
there exists an operator system $\cl S$ such that $\cl J = \cl S^{\perp}$. 
(Note that such subspaces were called trace-free non-commutative graphs in \cite{stahlke}.)
Let $\cl J\subseteq \cl L(H)$ be an operator anti-system. Recall \cite{km} that a 
\emph{strong independent set} for $\cl J$ is an orthonormal set $\{v_1,\dots,v_m\}$ of vectors in $H$ 
such that $v_iv_j^* \perp \cl J$ for all $i,j = 1,\dots,m$. 
It is clear that a subset $\{v_1,\dots,v_m\}$ is a strong independent set for $\cl J$ if and only if 
the projection onto its span is $\cl J^{\perp}$-full. 
The \emph{strong chromatic number} $\hat{\chi}(\cl J)$ of an operator anti-system $\cl J$ \cite{km}
is defined to be the smallest positive integer $k$ for which there exists an orthonormal basis of $H$ 
that can be partitioned into $k$ strong independent sets.

We now recall some classical graph parameters. Let $G = (X,E)$ be a graph on $d$ vertices. 
\begin{itemize}
\item[(i)] The \emph{independence number} $\alpha(G)$ of $G$ is the size of a maximal independent set in $G$;
\item[(ii)] The \emph{clique number} $\omega(G)$ of $G$ is the size of a maximal clique of $G$;
\item[(iii)] The \emph{chromatic number} $\chi(G)$ of $G$ is the smallest number of 
independent sets in $G$ with union $X$;
\item[(iv)] The \emph{fractional clique number} $\omega_{\rm f}(G)$ of $G$ is given by 
$$\omega_{\rm f}(G) = \max\left\{\sum_{x\in X} w_x : w_x \geq 0, x\in X, \  \sum_{x\in S}w_x \leq 1, \ \forall 
\mbox{ indep. set } S\right\}.$$
\end{itemize}

\noindent {\bf Remark. } Let $\cl S\subseteq M_d$ be an operator system. 
By Proposition \ref{l_gen} (i), 
\begin{equation}\label{eq_omef}
\omega_{\rm f}(\cl S) = \max\left\{\Tr(A) : A\in M_d^+, \Tr(AP)\leq 1, \ \forall \ \cl S\mbox{-abelian projection } P\right\}.
\end{equation}
Thus, the fractional clique number $\omega_{\rm f}(G)$ of a graph $G$ is given 
analogously to $\omega_{\rm f}(\cl S_G)$, but restricting the matrices $A$ in (\ref{eq_omef}) 
to be diagonal, and the projections $P$ to arise from independent sets of $G$. 
In fact, we have the following result, which is a direct consequence of Theorem \ref{apvp} and Corollary \ref{antiblockvp}.
Note that part (i) was noted in \cite{dsw}; we include it here for completeness.

\begin{corollary}\label{c_cnoncp}
Let $G = (X,E)$ be a graph on $d$ vertices. Then 

(i) \ \  $\alpha(G) = \alpha(\cl S_G)$;

(ii) \ $\omega(G) = \omega(\cl S_G) = \tilde{\omega}(\cl S_G)$;

(iii) $\omega_{\rm f}(G) = \omega_{\rm f}(\cl S_G)$;

(iv) \ $\omega_{\rm f}(G^c) = \kappa(\cl S_G) = \nph(\cl S_G)$.
\end{corollary}

\begin{remark}\label{r_cnoncp}
{\rm 
Let $G$ be a graph. It was shown in \cite{p_lect} that $\chi(G) = \chi(\cl S_G)$, 
and in \cite{km} that
$\hat{\chi}(\cl S_G^{\perp}) = \chi(G^c)$.
We complement these statements by the following proposition. 
} 
\end{remark}

Let 
$\cl S_n = {\rm span}\{e_ie_j^*, I_d : i \ne j \} \subseteq M_n$. 
(Note that $\cl S_n = \cl S_{E_n}^c$, where $E_n$ is the empty graph on $n$ vertices.)  
In \cite{lpt} it was shown that $\alpha(\cl S_2) = 1$ while, in \cite[Examples 4, 22]{km} -- that 
$\chi(\cl S_n) = \hat{\chi}(\cl S_n^{\perp}) = n$.
Here we extend these results by considering tensor products of operator systems of this type 
and identifying the values of some of the parameters introduced earlier. 
For $n_1,n_2,\ldots,n_m \in \bb{N}$, let 
$$\cl{S}_{n_1,\dots,n_m} = \cl{S}_{n_1} \otimes \cl{S}_{n_2} \otimes \dots  \otimes \cl{S}_{n_m}.$$

\begin{lemma} \label{S_nS_m}  
Let $u,v \in \mathbb{C}^{n_1n_2 \ldots n_m}$ be orthogonal vectors.

(i) \ If $uv^* \in \cl{S}_{n_1,\dots,n_m}^{\perp}$ then $uv^* = 0$;

(ii) If $uu^*, uv^* \in \cl{S}_{n_1,\dots,n_m}$ then $uv^* = 0$. 
\end{lemma}   

\begin{proof}
(i) 
Suppose that $uv^* \in \cl{S}_{n_1,\dots,n_m}^{\perp}$.
Let $m = 1$ and write 
$u = (u_i)_{i=1}^{n_1}$ and $v = (v_i)_{i=1}^{n_1}$. We have that 
$u_i \bar{v}_j = 0$ whenever $i \ne j$ and 
$\sum_{i=1}^{n_1} u_i \bar{v}_i = 0$. 
These conditions easily imply that $u = 0$ or $v = 0$. 

Proceeding by induction, suppose that the statement holds for some $m$. 
Note that
$$\cl{S}_{n_1,\dots,n_{m+1}}
= \left \{ \begin{pmatrix} S & S_{1,2} & \ldots & S_{1,n_1} \\  S_{2,1}& S & \ldots & S_{2,n_1} \\ \vdots & \vdots & \ddots & \vdots \\ S_{n_1,1} & S_{n_1,2} & \ldots & S \end{pmatrix} : 
S, S_{i,j} \in  \cl{S}_{n_2,\ldots,n_{m+1}}
\right\}.$$
Thus,
 $\cl{S}_{n_1,\dots,n_{m+1}}^{\perp}$ consists of all block matrices of the form 
$$\begin{pmatrix} D_{1,1} & D_{1,2} & \ldots & D_{1,n_1} \\  D_{2,1}& D_{2,2} & \ldots & D_{2,n_1} \\ \vdots & \vdots & \ddots & \vdots \\ D_{n_1,1} & D_{n_1,2} & \ldots & D_{n_1, n_1} \end{pmatrix},$$  
where  
\begin{equation} \label{perp1}  
D_{i,j} \in  \cl{S}_{n_2,\ldots,n_{m+1}}^{\perp}, \ \ \ i\neq j,
\end{equation}  
and 
\begin{equation} \label{perp2} 
\sum_{i=1}^{n_1} D_{i,i} \in  \cl{S}_{n_2,\ldots,n_{m+1}}^{\perp}.
\end{equation}   

Write  
$u = \begin{pmatrix} u_1 \\ \vdots \\ u_{n_1} \end{pmatrix}$ and $v=\begin{pmatrix} v_1 \\ \vdots \\ v_{n_1} \end{pmatrix}$,
where $u_i, v_i \in \mathbb{C}^{n_2 \ldots n_{m+1}}$, $i\in [n_1]$.
We have $uv^*= \left( u_iv_j^* \right)_{i,j=1}^{n_1}$ with $u_iv_j^* \in M_{n_2 \ldots n_{m+1}}.$ 
Assume that $u_i \neq 0$ for some $i\in [n_1]$.  
By \eqref{perp1} and the induction assumption, $v_j = 0$ whenever $j \ne i$.  
Now, by \eqref{perp2} and the induction assumption, $v_i = 0$; thus, $v = 0$.

(ii) 
Suppose that $u$ is a unit vector such that $uu^*, uv^* \in \cl{S}_{n_1,\dots,n_m}$. 
Write 
$$u = (u_{i_1,\dots,i_m})_{(i_1,\dots,i_m) \in [n_1]\times\cdots \times [n_m]}.$$ 
We will show that 
\begin{equation}\label{eq_manyn}
|u_{i_1,\dots,i_m}|^2 = \frac{1}{n_1\dots n_m}, \ \ \  (i_1,\dots,i_m) \in [n_1]\times\cdots \times [n_m],
\end{equation}
and $v = 0$. 
Letting $m = 1$ and writing 
$u = (u_i)_{i=1}^{n_1}$ and $v = (v_i)_{i=1}^{n_1}$, we have that 
$|u_i|^2 =  \frac{1}{n_1}$ for all $i \in [n_1]$.  
In addition,
$u_i \bar{v}_i = u_j \bar{v}_j$ for all $i,j$. Since $\langle u,v\rangle = 0$, 
we have that $u_i \bar{v}_i = 0$ for all $i\in [n_1]$. This shows that $v_i = 0$ for all $i\in [n_1]$, that is, $v = 0$. 

Proceeding by induction, suppose that the statement holds for some $m$
and write 
$u = \begin{pmatrix} u_1 \\ \vdots \\ u_{n_1} \end{pmatrix}$ and $v=\begin{pmatrix} v_1 \\ \vdots \\ v_{n_1} \end{pmatrix}$,
where $u_i, v_i \in \mathbb{C}^{n_2 \ldots n_{m+1}}$, $i\in [n_1]$.
We have that $u_i u_i^* = u_j u_j^*$ for all $i,j\in [n_1]$. 
Thus, $\|u_i\|^2 = \frac{1}{n_1}$ for all $i\in [n_1]$. 
The inductive assumption implies that $|u_{i_1,\dots,i_{m+1}}|^2 = \frac{1}{n_1\dots n_{m+1}}$ for all 
$(i_1,\dots,i_{m+1}) \in [n_1]\times\cdots [n_{m+1}]$.

On the other hand, 
$u_iv_i^* = u_jv_j^*$ for all $i,j\in [n_1]$ and hence $\langle u_i, v_i\rangle = 0$ for all $i \in [n_1]$. 
By the inductive assumption, $v_i = 0$ for all $i\in [n_1]$.
\end {proof}

\begin{proposition} \label{Snspaces}
Let $n_1,\dots,n_m\in \bb{N}$. Then

\begin{itemize}
\item[(i)] $\ap(\cl{S}_{n_1,n_2, \ldots, n_m}) = \{ A \in M_{n_1 n_2 \ldots n_m}: A \ge 0, \Tr(A) \le 1 \}$;

\item[(ii)] $\fp(\cl{S}_{n_1,n_2, \ldots, n_m}) = 
{\rm \her}( \overline{\conv}(\{uu^* :  u\in \bb{C}^{n_1 n_2 \ldots n_m}
\mbox{ satisfies } (\ref{eq_manyn})\}))$;

\item[(iii)] $\alpha (\cl{S}_{n_1,n_2 ,\ldots, n_m})=1$;


\item[(iv)] $\omega_f(\cl{S}_{n_1,n_2, \ldots, n_m}) = \chi(\cl{S}_{n_1,n_2, \ldots, n_m}) = n_1 \ldots n_m$;

\item[(v)] $\tilde{\omega}(\cl S_{n_1,n_2, \ldots, n_m}) = 1$;

\item[(vi)] $\nph(\cl S_{n_1,n_2, \ldots, n_m}) = n_1\dots n_m$; 

\item[(vii)] $\omega(\cl S_{n_1,n_2, \ldots, n_m}) \geq \min\{n_1,\dots,n_m\}$ and $\omega(\cl S_{n_1}) = n_1$. 
\end{itemize}
\end{proposition} 

\begin{proof} 
(i) and (ii) are immediate from Lemma \ref{S_nS_m}, (iii) is immediate from (i), and (v) from (ii). 

(iv) By (i), $I\in \ap(\cl{S}_{n_1,n_2, \ldots, n_m})^{\sharp}$, and hence $\omega_{\rm f} (\cl{S}_{n_1,n_2, \ldots, n_m}) = n_1 \ldots n_m$. 
The fact that $\chi(\cl{S}_{n_1,n_2, \ldots, n_m}) = n_1 \ldots n_m$ follows from Corollary \ref{c_ak}
and the fact that the value of $\chi(\cl S)$ does not exceed the dimension on which $\cl S$ acts.

(vi)
Set $d = n_1\ldots n_m$ for brevity.
By (ii), $I\in \fp(\cl S_{n_1,\ldots, n_m})^{\sharp}$; thus, $\nph(\cl S_{n_1,\ldots, n_m}) \geq d$. 
Fix a primitive $d$-th root of unity $\zeta$. 
For $k \in [d]$, let $u_k = \frac{1}{\sqrt{d}} \left(\zeta^{ki}\right)_{i=1}^d$; thus, $u_k$ is a unit vector in $\bb{C}^d$. 
By the proof of Lemma \ref{S_nS_m}, $u_k u_k^*\in \fp(\cl S_{n_1,\ldots,n_m})$. 
Thus, 
$$\frac{1}{d}  I_d = \frac{1}{d} \sum_{k=1}^d u_k u_k^* \in \fp(\cl S_{n_1,\ldots,n_m}).$$
It follows that if $T\in \fp(\cl S_{n_1,\ldots, n_m})^{\sharp}$ then $\Tr(T)\leq d$, and hence 
$\nph(\cl S_{n_1,\ldots, n_m}) = d$.

(vii) Assume, without loss of generality, that $n_1 = \min\{n_1,\dots,n_m\}$
and let $f_k : [n_1]\to [n_k]$ be an injective map, $k = 2,\dots,m$.
Then $\{e_{i_1, f_2(i_1),\dots,f_{m}(i_1)}$ $: i_1\in [n_1]\}$ is an $\cl S_{n_1,\dots,n_m}$-clique
and hence $\omega(\cl S_{n_1,n_2, \ldots, n_m}) \geq n_1$. 
On the other hand, $\omega(\cl S_{n_1})\leq n_1$ and the proof is complete. 
\end{proof}


\subsection{Inclusions between convex corners}\label{ss_icc}

We now prove our first sandwich theorem and list some of its consequences.

\begin{theorem}\label{th_apcp}
Let $\cl S$ be a non-commutative graph. Then 
$$\ap(\cl S) \subseteq \cp(\cl S)^{\sharp} \subseteq \fp(\cl S)^{\sharp}.$$
\end{theorem}

\begin{proof}
The second inclusion follows from Remark \ref{r_cpfpi} (ii). 
Let $\{\xi_i\}_{i=1}^k$ (resp. 
$\{\eta_p\}_{p=1}^m$) be an $\cl S$-independent set 
(resp. an $\cl S$-clique) and $P$ (resp. $Q$) be the projection onto its span.
It suffices to show that $\Tr(PQ) \leq 1$.
We have that 
$$\xi_i\xi_j^* \perp \eta_p\eta_q^* \ \mbox{ whenever } i\neq j \mbox{ and } p\neq q.$$
Thus, 
\begin{equation}\label{eq_ijpq}
i\neq j, \ p\neq q \ \Longrightarrow \ \xi_i\perp \eta_p \mbox{ or } \xi_j\perp \eta_q.
\end{equation}

For each $i \in [k]$, let 
$$\beta(i) = \{p\in [m] : \eta_p\perp \xi_i\}$$
and 
$$\alpha = \{i\in [k] : \beta(i)^c\neq \emptyset\}.$$
We distinguish three cases:

\smallskip

\noindent {\it Case 1.} $\alpha = \emptyset$. 
Then $\beta(i) = [m]$ for every $i\in [k]$. It follows that $Q\perp P$ and hence $\Tr(PQ) = 0 \leq 1$. 

\smallskip

\noindent {\it Case 2.} $|\alpha| = 1$, say $\alpha = \{i_0\}$. 
Then 
\begin{eqnarray*}
\Tr(PQ) 
& = & 
\sum_{i=1}^k \sum_{p=1}^m \Tr((\xi_i\xi_i^*)(\eta_p\eta_p^*)) 
= \sum_{i=1}^k \sum_{p=1}^m |\langle \xi_i,\eta_p\rangle|^2\\
& = & 
\sum_{i=1}^k \sum_{p\not\in \beta(i)} |\langle \xi_i,\eta_p\rangle|^2
= \sum_{p\not\in \beta(i_0)} |\langle \xi_{i_0},\eta_p\rangle|^2 \leq 1,
\end{eqnarray*}
because the family $\{\eta_p\}_{p=1}^m$ is orthonormal and $\|\xi_{i_0}\| = 1$.

\smallskip

\noindent {\it Case 3.} $|\alpha| > 1$.
By (\ref{eq_ijpq}), for every pair $(i,j)\in [k]\times [k]$ with $i\neq j$, we have 
$$\{(p,q) \in [m]\times [m] : p\neq q\} \subseteq \left(\beta(i)\times [m]\right)\cup \left([m]\times \beta(j)\right),$$
and hence 
\begin{equation}\label{eq_ijp}
\beta(i)^c \times \beta(j)^c \subseteq \{(p,p) : p \in [m]\} \ \mbox{ whenever } i\neq j. 
\end{equation}

Suppose that 
$|\beta(i)^c \times \beta(j)^c| > 1$ for some $i,j$ with $i\neq j$.
Then there are $p,p'$ such that $p\neq p'$ and 
$$p,p'\in \beta(i)^c, \ \ p,p'\in \beta(j)^c,$$
contradicting (\ref{eq_ijp}).
Thus, $|\beta(i)^c \times \beta(j)^c| \leq 1$ for all pairs $(i,j)$ with $i\neq j$. 
It follows that $|\beta(i)^c| = 1$ for every $i\in \alpha$. 
Write $\beta(i)^c = \{p_i\}$, $i\in \alpha$. 
Then $(p_i,p_j)\in \beta(i)^c \times \beta(j)^c$ for all $i,j\in \alpha$ with $i\neq j$. 
In view of (\ref{eq_ijp}), $p_i = p_j$ for all $i,j\in \alpha$.
Let $p_0$ be the common value of $p_i$, $i\in \alpha$;
then 
$$\Tr(PQ) = \sum_{i\in \alpha} |\langle \xi_i,\eta_{p_0}\rangle|^2\leq 1,$$
because the family $\{\xi_i\}_{i=1}^k$ is orthonormal and $\|\eta_{p_0}\| = 1$.
\end{proof}

\begin{corollary}\label{c_ak}
Let $\cl S$ be a non-commutative graph. Then 
\begin{equation}\label{eq_Sc}
\alpha(\cl S) = \omega(\cl S^c) \leq \kappa(\cl S) = \omega_{\rm f}(\cl S^c) \leq \nph(\cl S)
\leq \hat{\chi}(\cl S^{\perp}),
\end{equation} 
\begin{equation}\label{eq_Sc2}
\tilde{\omega}(\cl S) \leq \omega(\cl S) \leq \omega_{\rm f}(\cl S)\leq \chi(\cl S) \leq \hat{\chi}\left(\cl S\cap \{I\}^{\perp}\right)
\end{equation} 
and 
\begin{equation}\label{eq_Sc3}
\kappa(\cl S) \leq \chi(\cl S^c) \leq \hat{\chi}(\cl S^{\perp}).
\end{equation} 
\end{corollary}
\begin{proof}
By Proposition \ref{p_coic}, $\ap(\cl S) = \cp(\cl S^c)$ and hence $\alpha(\cl S) = \omega(\cl S^c)$
and $\omega_{\rm f}(\cl S^c) = \kappa(\cl S)$.
By Theorem \ref{th_apcp}, $\alpha(\cl S)\leq \kappa(\cl S) \leq \nph(\cl S)$.
Suppose that there exist $\cl S$-full projections $P_1,\dots,P_k$ with $I = \sum_{i=1}^k P_i$.
If $T\in \fp(\cl S)^{\sharp}$ then 
$$\Tr(T) = \sum_{i=1}^k \Tr(TP_i) \leq k;$$
thus, $\nph(\cl S) \leq \hat{\chi}(\cl S^{\perp})$. 

By Remark \ref{r_cpfpi} (ii), $\tilde{\omega}(\cl S) \leq \omega(\cl S)$. 
By Theorem \ref{th_apcp}, $\cp(\cl S)\subseteq \cp(\cl S)^{\sharp\sharp}\subseteq \ap(\cl S)^{\sharp}$ and 
hence $\omega(\cl S) \leq  \omega_{\rm f}(\cl S)$. 
An argument, similar to one in the previous paragraph shows that $\omega_{\rm f}(\cl S)\leq \chi(\cl S)$.
Replacing $\cl S$ with $\cl S^c$, we get $\kappa(\cl S)\leq \chi(\cl S^c)$. 

By Proposition \ref{p_coic} and 
Remark (ii) after Definition \ref{d_ab}, $\chi(\cl S^c)\leq \hat{\chi}(\cl S^{\perp})$. Replacing $\cl S$ with $\cl S^c$ 
and using Remark \ref{r_cc} we have
$\chi(\cl S) \leq \hat{\chi}\left(\cl S\cap \{I\}^{\perp}\right)$, and the proof is complete. 
\end{proof}

\noindent {\bf Remarks. (i) } 
By Corollary \ref{c_cnoncp} and Remark \ref{r_cnoncp}, 
the first and the last inequalities in (\ref{eq_Sc}), 
the second and the third inequalities in (\ref{eq_Sc2}),
and the first inequality in (\ref{eq_Sc3}) 
can be strict even in the case where $\cl S$ is a graph operator system. 
Let $\cl T$ be a non-commutative graph for which $\fp(\cl T) = \{0\}$ (see 
Remark \ref{r_cpfpi} (i)).
We have that $\nph(\cl T) = \hat{\chi}(\cl T^{\perp}) = \infty$ and $\tilde{\omega}(\cl T) = 0$; 
thus, by Remark \ref{r_cpfpi} (i), the middle inequality in (\ref{eq_Sc}), the first inequality in (\ref{eq_Sc2})
and the second inequality in (\ref{eq_Sc3}) can be strict.
In addition, $\hat{\chi}(\cl T^c\cap\{I\}^{\perp}) = \infty$, 
and hence the last inequality in (\ref{eq_Sc2}) can be strict. 

\smallskip

{\bf (ii) }
The inclusions in Theorem \ref{th_apcp} can be strict. 
For the first inclusion this follows for instance from 
the fact that $\alpha(\cl S)$ can be strictly smaller than $\kappa(\cl S)$ 
(see (i)). 
For the second inclusion, this follows from Remark \ref{r_cpfpi} and the fact that, for a convex corner 
$\cl A$, we have the identity $\cl A = \cl A^{\sharp\sharp}$.  
This non-trivial fact will be established in subsequent work \cite{btw2}. 

\smallskip

{\bf (iii) }
In view of the proof of Corollary \ref{c_ak}, 
the parameter $\nph(\cl S)$ can be thought of as a fractional version of the 
strong chromatic number $\hat{\chi}(\cl S^{\perp})$.


\section{The Lov\'asz corner}\label{s_lc}

Let $G = (X,E)$ be a graph on $d$ vertices. 
A family $(a_x)_{x\in X}$ of unit vectors in 
a finite dimensional complex Hilbert space is called an \emph{orthogonal labelling (o.l.)} of $G$ 
if 
$$x\not\simeq y \ \Rightarrow \ a_x \perp a_y.$$
Let 
$$\cl P_0(G) = \left\{\left(|\langle a_x,c\rangle|^2\right)_{x\in X} : (a_x)_{x\in X} \mbox{ is an o.l. of } G \mbox{ and } \|c\|\leq 1\right\},$$
viewed as a subset of $\cl D_X$.
Let
$\thab(G) = \cl P_0(G)^{\flat},$
and 
$$\theta(G) = \theta(\thab(G)) = \max\left\{\Tr(A) : A\in \thab(G)\right\}$$
be the \emph{Lov\'asz number} of $G$ \cite{lo}. 
We note that Lov\'{a}sz worked with real Hilbert spaces, but inspection of the proofs shows that the 
results in \cite{gls, knuth, lo} are true for complex Hilbert spaces as well.

In its strong form \cite{gls} (see also \cite{knuth}), the Sandwich Theorem states that 
\begin{equation}\label{eq_lst}
\vp(G) \subseteq \thab(G) \subseteq \fvp(G).
\end{equation}
The aim of this section is to introduce a non-commutative version of $\thab(G)$
and to establish a chain of inclusions, analogous to (\ref{eq_lst}).

\subsection{Definition and consistency}\label{ss_nlc}

Let $H$ and $K$ be finite dimensional Hilbert spaces 
and 
$\Phi : \cl L(H)\to \cl L(K)$ be a completely positive map. 
Then there exist operators 
$A_p : H\rightarrow K$, $p = 1,\dots,m$, such that
\begin{equation}\label{eq_kr}
\Phi(T) = \sum_{p=1}^m A_p T A_p^*, \ \ \ T\in \cl L(H);
\end{equation}
the form (\ref{eq_kr}) is called a \emph{Kraus representation} of $\Phi$, and $A_p$, $p = 1,\dots,m$
-- \emph{Kraus operators} of the representation. 
The operator system
$$\cl S_{\Phi} = {\rm span}\left\{A_p^*A_q : p,q\in [m]\right\}$$ 
is called the {\it non-commutative graph of} $\Phi$, and can be shown to be independent of
the Kraus representation (\ref{eq_kr}) of $\Phi$ (see \cite{dsw} and \cite[Corollary 2.23]{watrous}). 
We let
$\Phi^* : \cl L(K)\to \cl L(H)$ be the adjoint of $\Phi$, that is, the linear map given by 
$$\langle T,\Phi^*(S)\rangle = \langle \Phi(T),S\rangle, \ \ \ T\in \cl L(H), S\in \cl L(K).$$
The completely positive map $\Phi$ is called a {\it quantum channel} (\emph{q.c.}) 
if it is trace preserving; this is equivalent to the condition 
$\sum_{p=1}^m A_p^* A_p = I$. 
 
Let $\cl S\subseteq \cl L(H)$ be an operator system and 
$$\frak{C}(\cl S) = 
\left\{\Phi : \cl L(H) \to M_k \ : \ \Phi \mbox{ is a quantum channel with } \cl S_{\Phi} \subseteq \cl S, \ k\in \bb{N}\right\}.$$
Set 
$$\thet(\cl S) = \left\{T\in \cl L(H)^+ : \Phi(T)\leq I \mbox{ for every } \Phi\in \frak{C}(\cl S)\right\}$$
and
$$\frak{P}(\cl S) = \left\{\Phi^*(\sigma) : \Phi\in \frak{C}(\cl S), \sigma\geq 0, \Tr(\sigma)\leq 1\right\}.$$

\begin{lemma}\label{l_dualn}
Let $H$ be a finite dimensional Hilbert space and $\cl S\subseteq \cl L(H)$ be an operator system.
Then $\thet(\cl S)$ is a convex corner and 
$\thet(\cl S) = \frak{P}(\cl S)^{\sharp}$.
\end{lemma}

\begin{proof}
For $T\in \cl L(H)^+$, we have 
\begin{eqnarray*}
T\in \thet(\cl S) 
& \Leftrightarrow & 
\Phi(T)\leq I \mbox{ for all } \Phi \in \frak{C}(\cl S)\\
& \Leftrightarrow & 
\|\Phi(T)\|\leq 1 \mbox{ for all } \Phi \in \frak{C}(\cl S)\\
& \Leftrightarrow & 
\langle\Phi(T), \sigma\rangle \leq 1 \mbox{ for all } \Phi \in \frak{C}(\cl S)  \mbox{ and all } \sigma\geq 0, \Tr(\sigma) \leq 1  \\
& \Leftrightarrow & 
\langle T, \Phi^*(\sigma)\rangle \leq 1 \mbox{ for all } \Phi \in \frak{C}(\cl S)  \mbox{ and all } \sigma\geq 0, \Tr(\sigma) \leq 1 \\
& \Leftrightarrow & 
T\in \frak{P}(\cl S)^{\sharp}.
\end{eqnarray*}
Thus, $\thet(\cl S) = \frak{P}(\cl S)^{\sharp}$ and hence, by Remark \ref{r_trcc} (i), 
$\thet(\cl S)$ is a convex corner.
\end{proof}

For a non-commutative graph $\cl S$, we set
$$\theta(\cl S) = \theta(\thet(\cl S)),$$
and call $\theta(\cl S)$ the \emph{Lov\'{a}sz number} of $\cl S$.

Let $G = (X,E)$ be a graph. 
We will call a family $(P_x)_{x\in X}$ of projections acting on a finite dimensional Hilbert space 
a \emph{projective orthogonal labelling (p.o.l.)} of $G$ if 
$$x\not\simeq y \ \Rightarrow \ P_x P_y = 0.$$
Let 
$$\cl P(G) = \left\{\left(\Tr(P_x\rho)\right)_{x\in X} : (P_x)_{x\in X} \mbox{ is a p.o.l. of } G \mbox{ and } \rho \mbox{ is a state}\right\}.$$
Note that, if $(a_x)_{x\in X}$ is an orthogonal labelling of $G$ then 
the family $(a_xa_x^*)_{x\in X}$ is a projective orthogonal labelling of $G$. 
It follows that 
\begin{equation}\label{eq_pop}
\cl P_0(G) \subseteq \cl P(G).
\end{equation}

\begin{lemma}\label{l_dpol}
Let $G = (X,E)$ be a graph on $d$ vertices. Then $\cl P(G^c)\subseteq \cl P(G)^{\flat}$.
\end{lemma}

\begin{proof}
Let $(P_x)_{x\in X}$ (resp. $(Q_x)_{x\in X}$) be a projective orthogonal labelling of $G$ (resp. $G^c$)
acting on a Hilbert space $K_1$ (resp. $K_2$), 
and let $\rho$ (resp. $\sigma$) be a state on $K_1$ (resp. $K_2$). 
We have that 
$$P_x\otimes Q_x \perp P_y\otimes Q_y \ \ \mbox{whenever } x,y\in X,  \ x\neq y.$$
It follows that the operator $\sum_{x\in X} P_x\otimes Q_x$ is a contraction, and thus 
\begin{eqnarray*}
\sum_{x\in X} \Tr(P_x\rho) \Tr(Q_x\sigma) 
& = & 
\sum_{x\in X} \Tr\left((P_x\otimes Q_x)(\rho\otimes\sigma)\right)\\
& = & 
\Tr\left(\left(\sum_{x\in X} P_x\otimes Q_x\right)(\rho\otimes\sigma)\right)\leq 1.
\end{eqnarray*}
\end{proof}

\begin{lemma}\label{l_p0in}
Let $G = (X,E)$ be a graph on $d$ vertices. 
Then $\cl P_0(G)\subseteq \frak{P}(\cl S_G)$.
\end{lemma}

\begin{proof}
Let $(a_x)_{x\in X} \subseteq \bb{C}^k$ be an orthogonal labelling of $G$
and $\Phi : M_d\to M_k$ be the quantum channel defined by 
$$\Phi(S) = \sum_{x\in X} (a_x e_x^*) S (e_x a_x^*), \ \ \ S\in M_d.$$
If $x\not\simeq y$ then $(e_x a_x^*)(a_y e_y^*) = \langle a_y, a_x\rangle e_x e_y^* = 0$, 
and hence $\cl S_{\Phi} \subseteq \cl S_G$. 
Given a unit vector $c\in \bb{C}^k$, we have that 
$\Phi^*(cc^*) = \left(|\langle a_x,c\rangle|^2\right)_{x\in X}$, and the proof is complete.
\end{proof}

\begin{lemma}\label{l_diagin}
Let $G = (X,E)$ be a graph on $d$ vertices. Then 
$$\cl P(G)^{\flat}\subseteq \cl D_X \cap \thet(\cl S_G).$$
\end{lemma}

\begin{proof}
Let $\Phi : M_d\to M_k$ be a quantum channel with $\cl S_{\Phi}\subseteq \cl S_G$. 
Write 
$$\Phi(T) = \sum_{p=1}^m A_p T A_p^*, \ \ \ T\in M_d,$$
where $A_p : \bb{C}^d\to \bb{C}^k$ are linear operators such that 
$A_p^*A_q\in \cl S_G$ for all $p,q\in [m]$ and $\sum_{p=1}^m A_p^*A_p = I$. 
Let $a_{p,x} = A_pe_x$, $p\in [m]$, $x\in X$. Note that $a_{p,x}\in \bb{C}^k$ and set 
$Z_x = \sum_{p=1}^m a_{p,x}a_{p,x}^*$; thus, $Z_x\in M_k$, $x\in X$. 
Let $P_x$ be the projection onto the span of $\{a_{p,x} : p\in [m]\}$ and observe that 
\begin{equation}\label{eq_pxzx}
Z_x = P_x Z_x P_x, \ \ \ x\in X.
\end{equation}

Suppose that $x,y\in X$, $x\not\simeq y$. Then
$$\langle a_{q,y},a_{p,x}\rangle = \langle A_qe_y,A_pe_x\rangle = \langle A_p^*A_qe_y,e_x\rangle = 0, \ \ p,q\in [m].$$
It follows that the family $(P_x)_{x\in X}$ is a projective orthogonal labelling of $G$. 
On the other hand, 
\begin{equation}\label{eq_zx}
\|Z_x\| \leq 
\sum_{p=1}^m \|a_{p,x}a_{p,x}^*\|
=
\sum_{p=1}^m \|a_{p,x}\|^2
= \sum_{p=1}^m \langle A_p^*A_pe_x, e_x\rangle = 1.
\end{equation}
Relations (\ref{eq_pxzx}) and (\ref{eq_zx}) imply that 
\begin{equation}\label{eq_inpx}
Z_x\leq P_x, \ \ \ x\in X.
\end{equation}

Let $(t_x)_{x\in X}\in \cl P(G)^{\flat}$. Then  
\begin{eqnarray*}
\left\|\sum_{x\in X} t_x P_x\right\| 
& = & 
\max\left\{\Tr\left(\left(\sum_{x\in X} t_x P_x\right)\rho\right) : \rho \mbox{ a state on } \bb{C}^k\right\}\\
& = & 
\max\left\{\sum_{x\in X} t_x \Tr(P_x \rho) : \rho \mbox{ a state on } \bb{C}^k\right\} 
\leq 1.
\end{eqnarray*}
Since the operator $\sum_{x\in X} t_x P_x$ is positive, this implies that 
$\sum_{x\in X} t_x P_x \leq I$. 
Inequalities (\ref{eq_inpx}) now imply 
\begin{equation}\label{eq_less1}
\sum_{x\in X} t_x Z_x \leq I.
\end{equation}
Thus, letting $T = (t_x)_{x\in X}\in \cl D_X$, using (\ref{eq_less1}), we have
\begin{eqnarray*}
\Phi(T) 
& = & 
\sum_{p=1}^m A_pTA_p^* = 
\sum_{p=1}^m A_p \left(\sum_{x\in X} t_x e_x e_x^*\right) A_p^*\\
& = & 
\sum_{x\in X} t_x \sum_{p=1}^m (A_p e_x)(A_p e_x)^*
= 
\sum_{x\in X} t_x \sum_{p=1}^m a_{p,x} a_{p,x}^*
=
\sum_{x\in X} t_x Z_x \leq I.
\end{eqnarray*}
Therefore $T\in \thet(\cl S_G)$ and the proof is complete.
\end{proof}

\begin{theorem}\label{th_thetaeq}
Let $G = (X,E)$ be a graph on $d$ vertices. Then 
$$\thab(G) = \cl D_X\cap \thet(\cl S_G) = \Delta(\thet(\cl S_G)).$$
\end{theorem}

\begin{proof}
By \cite[Corollary 3.4]{gls}, (\ref{eq_pop}), Remark \ref{r_trcc} (ii) and Lemma \ref{l_dpol}, we have 
\begin{eqnarray*}
\cl P_0(G)^{\flat} 
& = & 
\thab(G) = \thab(G^c)^{\flat} = \cl P_0(G^c)^{\flat\flat} 
\subseteq \cl P(G^c)^{\flat\flat} \subseteq \cl P(G)^{\flat\flat\flat}\\
& = & 
\cl P(G)^{\flat} \subseteq \cl P_0(G)^{\flat}.
\end{eqnarray*}
Thus, $\thab(G) = \cl P(G)^{\flat}$ and hence, by Lemma \ref{l_diagin}, 
\begin{equation}\label{eq_onin}
\thab(G) \subseteq \cl D_X\cap \thet(\cl S_G).
\end{equation}

Clearly, $\cl D_X\cap \thet(\cl S_G) \subseteq \Delta(\thet(\cl S_G))$.
Let $T\in \thet(\cl S_G)$ and suppose that $\Phi : M_d\to M_k$ is a quantum channel with $\cl S_{\Phi}\subseteq \cl S_G$. 
Let
$$\Phi(S) = \sum_{p=1}^m A_p S A_p^*, \ \ \ S\in M_d,$$
be a Kraus representation of $\Phi$.
Set $A_{p,x} = A_p (e_x e_x^*)$, and note that, since $\cl S_G$ is a $\cl D_X$-bimodule, we have that 
$A_{p,x}^*A_{q,y} = (e_x e_x^*) A_p^*A_q (e_y e_y^*)\in \cl S_G$. 
In addition,
$$\sum_{p=1}^m \sum_{x\in X} A_{p,x}^* A_{p,x} = \sum_{p=1}^m \sum_{x\in X}  (e_x e_x^*)A_p^* A_p (e_x e_x^*)
= \sum_{x\in X} e_x e_x^* = I.$$
Thus, the map
$\Psi : M_d\to M_k$, given by 
$$\Psi(S) = \sum_{p=1}^m \sum_{x\in X} A_{p,x} S A_{p,x}^*, \ \ \ S\in M_d,$$
is a quantum channel with $\cl S_{\Psi}\subseteq \cl S_G$. 
Hence
$$\Phi(\Delta(T)) = \sum_{p=1}^m A_p \left(\sum_{x\in X} (e_x e_x^*) T (e_x e_x^*)\right) A_p^* 
= \Psi(T) \leq I.$$
It follows that $\Delta(T)\in \thet(\cl S_G)$.
By Lemmas \ref{l_dualn} and \ref{l_p0in}, $\thet(\cl S_G)\subseteq \cl P_0(G)^{\sharp}$. 
It now follows that $\Delta(T)\in  \thab(G)$.
The proof is complete.
\end{proof}

In view of Theorem \ref{th_thetaeq}, $\thet(\cl S)$ can be thought of as a 
non-commutative version of $\thab(G)$.

\begin{corollary}\label{c_thq}
Let $G = (X,E)$ be a graph on $d$ vertices. Then $\theta(G) = \theta(\cl S_G)$. 
\end{corollary}

\begin{proof}
By Theorem \ref{th_thetaeq}, $\thab(G)\subseteq \thet(\cl S_G)$, and hence $\theta(G)\leq \theta(\cl S_G)$. 
Since $\thet(\cl S_G)$ is compact, there exists $T\in \thet(\cl S_G)$ such that $\Tr(T) = \theta(\cl S_G)$. 
By Theorem \ref{th_thetaeq}, $\Delta(T) \in \thab(G)$, and hence 
$\Tr(T) = \Tr(\Delta(T)) \leq \theta(G)$. 
\end{proof}

\subsection{The second sandwich theorem}\label{ss_lst}
We now establish a chain of inclusions generalising the Sandwich Theorem (\ref{eq_lst}) to the non-commutative setting.

\begin{theorem}\label{th_sand}
Let $H$ be a finite dimensional Hilbert space and $\cl S\subseteq \cl L(H)$ be an operator system.
Then 
$$\ap(\cl S) \subseteq \thet(\cl S) \subseteq \fp(\cl S)^{\sharp}.$$
\end{theorem}

\begin{proof}
Let $P$ be an $\cl S$-abelian projection, and suppose that 
$\{\xi_i\}_{i=1}^k$ is an $\cl S$-independent set of (unit) vectors that spans $PH$. 
Fix $\Phi\in \frak{C}(\cl S)$ with Kraus operators $A_1,\dots,A_m$, and let $i,j\in [k]$ with $i\neq j$. 
Then 
\begin{eqnarray*}
\Tr\left(\Phi(\xi_i\xi_i^*) \Phi(\xi_j\xi_j^*)\right) 
& = & 
\sum_{p,q = 1}^m \Tr\left(A_p(\xi_i\xi_i^*)A_p^* A_q(\xi_j\xi_j^*)A_q^*\right)\\
& = & 
\sum_{p,q = 1}^m \Tr\left((A_p\xi_i)(A_p\xi_i)^* ((A_q\xi_j)(A_q\xi_j)^*)\right)\\
& = & 
\sum_{p,q = 1}^m \left| \left\langle A_q\xi_j, A_p\xi_i \right\rangle\right|^2
= 
\sum_{p,q = 1}^m \left|\left\langle A_p^*A_q\xi_j, \xi_i\right\rangle\right|^2 = 0,
\end{eqnarray*}
since $\cl S_{\Phi}\subseteq \cl S$ while $\xi_i\xi_j^* \in \cl S^{\perp}$.
Since $\Phi(\xi_i\xi_i^*)$ and $\Phi(\xi_j\xi_j^*)$ are positive operators, we conclude that there exist 
mutually orthogonal projections $Q_1,\dots,Q_k$ such that 
$\Phi(\xi_i\xi_i^*) = Q_i\Phi(\xi_i\xi_i^*)Q_i$ for each $i \in [k]$. 
Since $\|\xi_i\| = 1$ and $\Phi$ is trace preserving,
$\Tr(\Phi(\xi_i\xi_i^*)) = 1$;
in particular, $\left\|\Phi(\xi_i\xi_i^*)\right\|\leq 1$. 
It now follows that 
$$\left\|\Phi(P)\right\| 
= \left\|\sum_{i=1}^k \Phi(\xi_i\xi_i^*)\right\| = \max_{i=1,\dots,k} \left\|\Phi(\xi_i\xi_i^*)\right\|  \leq 1,$$
that is, $P\in \thet(\cl S)$.
Since $\ap(\cl S)$ is generated, as a convex corner, by the $\cl S$-abelian projections, 
using Lemma \ref{l_dualn}, we conclude that $\ap(\cl S)\subseteq \thet(\cl S)$. 

Now suppose that $Q$ is an $\cl S$-full projection. Let 
$(\eta_j)_{j=1}^k$ be an orthonormal basis for the range of $Q$; then $\eta_i\eta_j^*\in \cl S$ for all $i,j\in [k]$.
Let $\eta$ be a unit vector with $\eta = Q\eta$ and 
$\Phi : \cl L(H)\to \cl L(H)$ be the quantum channel given by 
$$\Phi(T) = Q^{\perp} T Q^{\perp} + \Tr(TQ)\eta\eta^*, \ \ \ T\in \cl L(H).$$
For any $T\in \cl L(H)$ we have 
$$\langle T, \Phi^*(\eta\eta^*)\rangle = 
\langle \Phi(T), \eta\eta^*\rangle = \Tr(TQ);$$
thus, 
\begin{equation}\label{eq_eta}
\Phi^*(\eta\eta^*) = Q. 
\end{equation}
Note that 
$$\Phi(T) = Q^{\perp} T Q^{\perp} + \sum_{j=1}^k (\eta\eta_j^*) T (\eta_j\eta^*), \ \ \ T\in \cl L(H),$$
and that 
$$(\eta_i\eta^*)(\eta\eta_j^*) = \eta_i\eta_j^* \in \cl S, \ \ \ i, j\in [k].$$
It follows that $\cl S_{\Phi}\subseteq \cl S$. 
Now (\ref{eq_eta}) implies that $Q\in \frak{P}(\cl S)$. 
Since $Q$ is an arbitrary $\cl S$-full projection, by Lemma \ref{l_dualn},
$\thet(\cl S) = \frak{P}(\cl S)^{\sharp} \subseteq \fp(\cl S)^{\sharp}.$
\end{proof}

The classical Lov\'asz Sandwich Theorem
states that the chain of inequalities 
$$\alpha(G)\leq \theta(G) \leq \omega_{\rm f}(G^c)$$
holds for a graph $G$ (see \cite{knuth}). 
The next corollary, which is immediate from Theorem \ref{th_sand}, provides a non-commutative version.

\begin{corollary}\label{c_ast}
If $\cl S$ is a non-commutative graph then
$\alpha(\cl S)\leq \theta(\cl S) \leq \nph(\cl S).$
\end{corollary}


\section{Another quantisation of $\theta(G)$}\label{s_aqthe}

Let $\Phi : M_d\to M_k$ be a quantum channel. 
The \emph{one-shot classical Shannon capacity} of $\Phi$ was introduced in \cite{dsw}
as the maximal number of pure states that can be transmitted via $\Phi$ so that their 
images are perfectly distinguishable. As was pointed out in \cite{dsw}, 
it coincides with the independence number $\alpha(\cl S)$ of the non-commutative confusability graph $\cl S$ 
of $\Phi$. 
The \emph{classical Shannon capacity} \cite{dsw} of $\Phi$
is, on  the other hand, defined by setting 
$$c_0(\Phi) = \lim_{n\to\infty} \sqrt[n]{\alpha(\cl S^{\otimes n})},$$
where $\cl S^{\otimes n} = \underbrace{\cl S\otimes\cdots \otimes \cl S}_{n \ {\rm times}}$.
We note that it depends only on $\cl S$; thus, one may talk without ambiguity about the 
Shannon capacity of a non-commutative graph $\cl S$ and denote it by $c_0(\cl S)$. 

In the case $\cl S$ is the operator system of a graph $G$, we have that $c_0(\cl S)$ coincides with the 
Shannon capacity $c_0(G)$ of $G$. 
The Lov\'asz number of $G$ is in this case an upper bound of $c_0(G)$. 
We do not know if the inequality $c_0(\cl S)\leq \theta(\cl S)$ holds for general non-commutative graphs $\cl S$. 
However, $\theta(G)$ 
has several equivalent expressions \cite{gls, knuth, lo}; in particular, 
we have that
\begin{equation}\label{eq_ano}
\theta(G) = \min\left\{\max_{x\in X} \frac{1}{\left|\langle a_x,c\rangle\right|^2} : \|c\| = 1, (a_x)_{x\in X} \mbox{ an o.l. of } G\right\}.
\end{equation}
We will now consider a non-commutative version of the latter expression and show that it leads to a parameter 
that bounds from above the Shannon capacity of the corresponding non-commutative graph. 
Namely, for a non-commutative graph $\cl S$, we set
$$\hat{\theta}(\cl S) = \inf\left\{\left\|\Phi^*(\sigma)^{-1}\right\| : 
\sigma \geq 0, \Tr(\sigma)\leq 1, \Phi \in \frak{C}(\cl S), \Phi^*(\sigma) \mbox{ invertible}\right\}.$$

\begin{theorem}\label{th_eta}
Let $H$ be a finite dimensional Hilbert space and $\cl S\subseteq \cl L(H)$ be an operator system. 
Then

\smallskip

(i) \ $\hat{\theta}(\cl S)^{-1} = \sup\left\{\inf\left\{\|\Phi(\rho)\| : \rho \mbox{ a state on } H\right\} : \Phi \in \frak{C}(\cl S)\right\}$;

(ii) $\theta(\cl S)^{-1} =  \inf\left\{\sup\left\{\|\Phi(\rho)\| : \Phi \in \mathfrak C(\cl S)\right\} : \rho \mbox{ a state on } H\right\}$.
\end{theorem}

\begin{proof}
(i) 
For an operator $A\in \cl L(H)^+$, write $\lambda_{\min}(A)$ for the smallest eigenvalue of $A$. 
Using the von Neumann minimax theorem, we have 
\begin{eqnarray*}
& & 
\hat{\theta}(\cl S)^{-1}\\
& = & 
\inf\left\{\|\Phi^*(\sigma)^{-1}\| : \sigma \mbox{ a state with } \lambda_{\min}(\Phi^*(\sigma)) > 0, \Phi\in \frak{C}(\cl S)\right\}^{-1}\\
& = & 
\inf\left\{\lambda_{\min}^{-1}(\Phi^*(\sigma)) : \sigma \mbox{ a state with } \lambda_{\min}(\Phi^*(\sigma)) > 0, 
\Phi\in \frak{C}(\cl S)\right\}^{-1}\\
\end{eqnarray*}
\begin{eqnarray*}
& = & 
\sup\left\{\lambda_{\min}(\Phi^*(\sigma)): \sigma \mbox{ a state with } \lambda_{\min}(\Phi^*(\sigma)) > 0, 
\Phi\in \frak{C}(\cl S)\right\}\\
& = & 
\sup\left\{\lambda_{\min}(\Phi^*(\sigma)): \sigma \mbox{ a state}, 
\Phi\in \frak{C}(\cl S)\right\}\\
& = & 
\sup\left\{\max\left\{\min\{\langle \Phi^*(\sigma)\xi,\xi\rangle : \|\xi\| = 1\} : \sigma \mbox{ a state}\right\} : \Phi\in \frak{C}(\cl S)\right\}\\
& = & 
\sup\left\{\max\left\{\min\{\langle \Phi^*(\sigma), \rho\rangle : \rho \mbox{ a state}\} : 
\sigma \mbox{ a state}\right\} : \Phi\in \frak{C}(\cl S)\right\}\\
& = & 
\sup\left\{\max\left\{\min\{\langle \sigma, \Phi(\rho)\rangle : \rho \mbox{ a state}\} : 
\sigma \mbox{ a state}\right\} : \Phi\in \frak{C}(\cl S)\right\}\\
& = & 
\sup\left\{\min\left\{\max\{\langle \sigma, \Phi(\rho)\rangle : \sigma \mbox{ a state}\} : 
\rho \mbox{ a state}\right\} : \Phi\in \frak{C}(\cl S)\right\}\\
& = & 
\sup\left\{\min\left\{\|\Phi(\rho)\| : \rho \mbox{ a state}\right\} : \Phi\in \frak{C}(\cl S)\right\}.
\end{eqnarray*}

(ii) 
Since the operator $\frac{1}{\Tr (T)}T$ is a state for each non-zero $T \in M_d^+$, we have
\begin{eqnarray*} 
\theta (\cl S)^{-1} 
& = & 
\sup\left\{\sup\{\lambda > 0 : \lambda \rho \in \thet (\cl S)\} : \rho \mbox{ a state}\right\}^{-1}\\
& = & 
\sup \left\{\sup\left\{\lambda > 0 : \|\Phi(\lambda \rho)\| \le 1 \mbox{ for all } \Phi \in \mathfrak C(\cl S)\right\} : \rho \mbox{ a state}
\right\}^{-1}\\
& = & 
\sup \left\{\sup\left\{\lambda > 0 : \lambda \leq \|\Phi(\rho)\|^{-1} \mbox{ for all } \Phi \in \mathfrak C(\cl S)\right\} : \rho \mbox{ a state}
\right\}^{-1}\\
& = & 
\sup \left\{\inf \left\{\|\Phi(\rho)\|^{-1} : \Phi \in \mathfrak C(\cl S)\right\}: \rho \mbox{ a state}\right\}^{-1}\\
& = & 
\inf \left\{\sup\left\{\|\Phi(\rho)\| : \Phi \in \mathfrak C(\cl S)\right\} : \rho \mbox{ a state}\right\}.
\end{eqnarray*}  
\end{proof}

\noindent {\bf Remark. } 
By compactness, the infimum appearing in part (i) of Theorem \ref{th_eta} is a minimum;
it is not however clear whether the supremum in this expression is attained. 

\medskip

Recall \cite{lpt} that the \emph{quantum subcomplexity} $\beta(\cl S)$ of a non-commutative graph 
$\cl S \subseteq \cl L(H)$ is defined by letting
$$\beta(\cl S) = \min\left\{k \in \bb{N} : \mbox{ there exists q.c. } \Phi : \cl L(H)\to M_k \mbox{ with }
\Phi\in \frak{C}(\cl S)\right\}.$$

\begin{theorem}\label{p_theta1}
Let $H$ be a Hilbert space with $\dim(H) = d$ and $\cl S \subseteq \cl L(H)$ be an operator system. 
Then 
$$d \inf\left\{\|\Phi( I_d) \|^{-1} : \Phi \in \mathfrak C(\cl S)\right\} 
\leq \theta(\cl S) \leq \hat{\theta}(\cl S) \leq \beta(\cl S) \leq d.$$
\end{theorem}

\begin{proof}
The first inequality follows from Theorem \ref{th_eta} (ii) by taking $\rho = \frac{1}{d} I_d$. 
The second inequality is immediate from Theorem \ref{th_eta}.
If $\Phi : \cl L(H) \to M_k$ is a quantum
channel in $\frak{C}(\cl S)$ then, letting $\sigma = \frac{1}{k}I_k$ we have 
$\left\|\Phi^*(\sigma)^{-1}\right\| = k$. Thus, $\hat{\theta}(\cl S) \leq \beta(\cl S)$. 
The last inequality, as pointed out in \cite{lpt}, follows by noting that the identity channel 
belongs to $\frak{C}(\cl S)$. 
\end{proof}

\begin{proposition}\label{p_quan}
Let $G = (X,E)$ be a graph. Then $\hat{\theta}(\cl S_G) = \theta(G)$.
\end{proposition}

\begin{proof}
By Theorems \ref{th_thetaeq} and \ref{p_theta1}, 
\begin{equation}\label{eq_ch1}
\theta(G) = \theta(\cl S_G) \leq \hat{\theta}(\cl S_G).
\end{equation}
Let $(a_x)_{x\in X} \subseteq \bb{C}^k$ be an orthogonal labelling of $G$
and $\Phi : M_d\to M_k$ be the quantum channel defined in the proof of Lemma \ref{l_p0in}. 
Let $c$ be a unit vector in $\bb{C}^k$ such that $\langle a_x,c\rangle \neq 0$ for all $x\in X$. 
We have that 
$$\left\|\Phi^*(cc^*)^{-1}\right\| = \max_{x\in X} \frac{1}{\left|\langle a_x,c\rangle\right|^2}.$$
Taking the infimum over all orthogonal representations of $G$ and unit vectors $c$
and using (\ref{eq_ano}),
we conclude that $\hat{\theta}(\cl S_G)\leq \theta(G)$. Together with (\ref{eq_ch1}), this completes the proof. 
\end{proof}

\noindent {\bf Remark. } It follows from Proposition \ref{p_quan} that the second inequality 
in Theorem \ref{p_theta1} can be strict; indeed, $\beta(\cl S)$ is an integer while $\theta(\cl S)$ can 
be fractional (for example, if $C_5$ is the $5$-cycle then $\theta(C_5) = \sqrt{5}$).

\begin{proposition}\label{p_subm}
Let $H_1$ and $H_2$ be finite dimensional Hilbert spaces and 
$\cl S_1\subseteq \cl L(H_1)$ and $\cl S_2 \subseteq \cl L(H_2)$ be operator systems. Then 
$\hat{\theta}(\cl S_1\otimes\cl S_2) \leq \hat{\theta}(\cl S_1) \hat{\theta}(\cl S_2)$.
\end{proposition}

\begin{proof}
Let $\epsilon > 0$, let $\sigma_i \in \cl L(H_i)^+$ be an operator with $\Tr(\sigma_i) \leq 1$
and $\Phi_i^*(\sigma_i)$ invertible, 
and let 
$\Phi_i : \cl L(H_i)\to \cl L(K_i)$ be a quantum channel with $\cl S_{\Phi_i}\subseteq \cl S_i$, $i = 1,2$, 
such that 
$$\left\|\Phi_i^*(\sigma_i)^{-1}\right\| \leq \hat{\theta}(\cl S_i) + \epsilon, \ \ \ i = 1,2.$$
Then $\Phi_1\otimes\Phi_2 : \cl L(H_1\otimes H_2) \to \cl L(K_1\otimes K_2)$ is a quantum channel
with $\cl S_{\Phi_1\otimes\Phi_2}\subseteq \cl S_1\otimes \cl S_2$.
In addition,
$(\Phi_1\otimes\Phi_2)^*(\sigma_1\otimes\sigma_2) = 
\Phi_1^*(\sigma_1) \otimes \Phi_2^*(\sigma_2)$
is invertible and 
\begin{eqnarray*}
\hat{\theta}(\cl S_1 \otimes \cl S_2) 
& \leq & \left\|(\Phi_1\otimes\Phi_2)^*(\sigma_1\otimes\sigma_2)^{-1}\right\|
=  \left\|\Phi_1^*(\sigma_1)^{-1}\right\| \left\|\Phi_2^*(\sigma_2)^{-1}\right\|\\
& \leq &  
(\hat{\theta}(\cl S_1) + \epsilon)(\hat{\theta}(\cl S_2) + \epsilon).
\end{eqnarray*}
The conclusion follows by letting $\epsilon \to 0$.
\end{proof}

Corollary \ref{c_ast}, Theorem \ref{p_theta1} and Proposition \ref{p_subm} have the following immediate consequence.
In view of Theorem \ref{p_theta1}, it improves the bound on the Shannon capacity by 
the parameter $\beta$, established in \cite{lpt}.

\begin{corollary}\label{c_boundSh}
Let $\cl S$ be a non-commutative graph. Then $c_0(\cl S) \leq \hat{\theta}(\cl S)$. 
\end{corollary}

Let $\vartheta(\cl S)$ be the quantisation 
of the Lov\'{a}sz number defined in \cite{dsw}, namely 
$$\vartheta(\cl S) = \max\{\|I + T\| : T\in \cl S^{\perp}, I + T \geq 0\},$$
and $\tilde{\vartheta}(\cl S) = \sup_{m\in \bb{N}} \vartheta(M_m(\cl S))$ be its complete version. 
It was shown in \cite{dsw} that $\tilde{\vartheta}(\cl S)$ is a bound on the Shannon capacity of $\cl S$.
The next examples imply that $\hat{\theta}$ can be genuinely better than $\tilde{\vartheta}$.

\medskip

\noindent {\bf Examples. (i) }
Let $d\in \bb{N}$. Suppose that $\Phi$ is a quantum channel with (non-zero) Kraus operators $A_1,\dots, A_m$
and confusability graph $\bb{C}I_d$. Then $A_i^*A_i = \lambda_i I$ for some $\lambda_i > 0$, and hence
the operator $V_i := \frac{1}{\sqrt{\lambda_i}} A_i$ is an isometry. 
Thus, 
$$\sum_{i=1}^m A_i A_i^* = \sum_{i=1}^m \lambda_i V_i V_i^* \leq \sum_{i=1}^m \lambda_i I = I.$$
Therefore, $I_d\in \thet(\bb{C} I_d)$ and hence $\theta(\bb{C} I_d) \geq d$. Together with Theorem \ref{p_theta1}
this shows that $\theta(\bb{C} I_d) = \hat{\theta}(\bb{C} I_d) = d$.
Note, on the other hand, that $\tilde{\vartheta}(\bb{C}I_d) = d^2$ \cite[(7)]{dsw},
showing that the ratio $\frac{\hat{\theta}(\cl S)}{\tilde{\vartheta}(\cl S)}$,
where $\cl S$ is a non-commutative graph, can be arbitrarily small.

\smallskip

{\bf (ii)} 
Let $\cl S_d = {\rm span} \{I_d, e_ie_j^* : i\neq j\} \subseteq M_d$.
It was shown in \cite[Theorem V.2]{lpt} that 
$\beta(\cl S_d\otimes \cl S_{d^2}) \leq d^2$; on the other hand, \cite{dsw} easily implies 
that $\vartheta(\cl S_d\otimes \cl S_{d^2}) \geq d^3$ (see \cite[Theorem V.2]{lpt}). 
It follows that the ratio
$\frac{\hat{\theta}(\cl S)}{\vartheta(\cl S)}$,
where $\cl S$ is a non-commutative graph, can also be arbitrarily small.


\section{Further properties}\label{s_cp}

In this section, we study the dependence of some of the parameters we introduced on the
operator system. Our main focus is on $\theta$, 
but we also point out some auxiliary results for the other parameters.

\subsection{Monotonicity}\label{ss_m}

Let $\cl S \subseteq \cl L(H)$ and $\cl T \subseteq \cl L(K)$ be non-commutative graphs. 
A \emph{homomorphism} from $\cl S$ into $\cl T$ \cite{stahlke} 
is a quantum channel $\Gamma  : \cl L(H)\to \cl L(K)$ that has a Kraus representation 
$\Gamma(S) = \sum_{i=1}^m A_i S A_i^*$, such that 
$$A_i^* \cl T A_j \subseteq \cl S, \ \ \ i,j = 1,\dots,m.$$
If there exists a homomorphism from $\cl S$ into $\cl T$, we write $\cl S \to \cl T$. 
It was shown in \cite{stahlke} that if $G$ and $H$ are graphs then $\cl S_G\to \cl S_H$ if and only if 
there exists a graph homomorphism from $G$ to $H$.

\begin{proposition}\label{p_homom}
Let $\cl S$ and $\cl T$ be non-commutative graphs. If $\cl S\to \cl T$ then $\theta(\cl S) \leq \theta(\cl T)$
and $\hat{\theta}(\cl S) \leq \hat{\theta}(\cl T)$.
\end{proposition}

\begin{proof}
Let $\Gamma$ be a homomorphism from $\cl S$ into $\cl T$.
If $\Psi\in \frak{C}(\cl T)$ then $\Psi \circ \Gamma\in \frak{C}(\cl S)$.
Letting $S\in \thet(\cl S)$ be such that $\Tr(S) = \theta(\cl S)$ we have 
that $\Gamma(S)\in \cl L(K)^+$ and 
$(\Psi \circ \Gamma)(S) \leq I$, showing that $\Gamma(S) \in \thet(\cl T)$. 
Thus, $\theta(\cl T) \geq \Tr(\Gamma(S)) = \theta(\cl S).$
In addition, 
$$\min\{\|(\Psi\circ \Gamma)(\sigma)\| : \sigma \mbox{ a state on } H\}\\
\geq 
\min\{\|\Psi(\rho)\| : \rho \mbox{ a state on } K\}.$$
By Theorem \ref{th_eta}, 
$\hat{\theta}(\cl S) \leq \hat{\theta}(\cl T)$.
\end{proof}

\begin{lemma}\label{l_smms}
Let $H$ be a finite dimensional Hilbert space, 
$\cl S\subseteq \cl L(H)$ be an operator system and $m\in \bb{N}$. 
Then $M_m(\cl S) \to \cl S$ and $\cl S\to M_m(\cl S)$. 
\end{lemma}

\begin{proof}
For $i \in [m]$, 
let $V_i : H\to H^{m}$ be the operator given by $V_i(\xi) = (\xi_k)_{k=1}^m$, where $\xi_i = \xi$ and 
$\xi_k = 0$ if $k\neq i$.
Then $\sum_{i=1}^m V_i V_i^* = I$, and hence the map $\Gamma : \cl L(H^m)\to \cl L(H)$ given by 
$\Gamma(S) = \sum_{i=1}^m V_i^* S V_i$, is a quantum channel. 
Clearly, $V_i \cl S V_j^* \subseteq M_m(\cl S)$ for all $i,j \in [m]$; in other words, $M_m(\cl S) \to \cl S$. 

Let $\Lambda : \cl L(H)\to \cl L(H^m)$ be the channel given by $\Lambda(S) = \frac{1}{m} S\otimes I_m$; 
thus, $\Lambda(S) = \frac{1}{m} \sum_{i=1}^m V_i S V_i^*$, $S\in \cl L(H)$.
Moreover, $V_i^* M_m(\cl S) V_j = \cl S$ for all $i,j \in [m]$. 
It follows that $\cl S\to M_m(\cl S)$. 
\end{proof}

\begin{corollary}\label{c_stability}
Let $H$ be a finite dimensional Hilbert space and $\cl S \subseteq \cl L(H)$ be an operator system. 
Then $\theta(\cl S) = \theta(M_m(\cl S))$ and $\hat{\theta}(\cl S) = \hat{\theta}(M_m(\cl S))$ for every $m\in \bb{N}$.
\end{corollary}

\begin{proof}
Immediate from Proposition \ref{p_homom} and Lemma \ref{l_smms}. 
\end{proof}

\noindent {\bf Remark.} 
Let $H$ (resp. $K$) be a finite dimensional Hilbert space and $\cl S\subseteq \cl L(H)$ 
(resp. $\cl T\subseteq \cl L(K)$) be an operator system. If $\cl S\to \cl T$ then $\alpha(\cl S)\leq \alpha(\cl T)$.
Indeed, suppose that $\Gamma : \cl L(H)\to \cl L(K)$ is a homomorphism from $\cl S$ to $\cl T$ 
with Kraus operators $A_1,\dots,A_m$, and that $\{\xi_i\}_{i=1}^k \subseteq H$ is an $\cl S$-independent set. 
Let $V$ be the column isometry $(A_i)_{i=1}^m$; then $\{V\xi_i\}_{i=1}^k$ is an $M_m(\cl T)$-independent set, 
and hence $\alpha(\cl S) \leq \alpha(M_m(\cl T))$. The claim follows from the fact that 
$\alpha(M_m(\cl T)) = \alpha(\cl T)$ \cite[Remark IV.6 (i)]{lpt}. 
By Corollary \ref{c_ak},
if $\cl S^c\to \cl T^c$ then $\omega(\cl S)\leq \omega(\cl T)$, and 
$\omega(M_m(\cl T)) = \omega(\cl T)$ for any $\cl T$. 

\subsection{Equivalent expressions and grading}\label{ss_eeg}

Let $H$ be a $d$-dimensional Hilbert space and $\cl S$ be a non-commutative graph acting on $H$.
We call a linear map $\Phi : M_d\to M_k$ a \emph{subchannel} if it is completely positive and trace decreasing. 
Note that $\Phi$ is a subchannel if and only if its dual $\Phi^* : M_k\to M_d$
is subunital in the sense that $\Phi^*(I)\leq I$, if and only if the Kraus operators $A_1\dots,A_m : \bb{C}^d\to\bb{C}^k$
satisfy the relation $\sum_{i=1}^m A_i^*A_i \leq I$. 
Set 
$$\frak{C}_{{\rm sub}}(\cl S) = \left\{\Phi : M_d\to M_k \ : \ k\in \bb{N}, \ \Phi \mbox{ a subchannel with }
\cl S_{\Phi}\subseteq \cl S\right\}.$$
Let 
$$\thet\mbox{}_{\sub}(\cl S) = \{T\in \cl L(H)^+ : \Phi(T)\leq I \mbox{ whenever } \Phi\in \frak{C}_{\sub}(\cl S)\},$$
$$\frak{P}_{\sub}(\cl S) = \{\Phi^*(\sigma) : \Phi\in \frak{C}_{\sub}(\cl S), \sigma\geq 0, \Tr(\sigma)\leq 1\},$$
and $\theta_{\sub}(\cl S) = \max\{\Tr(A) : A\in \thet_{\sub}(\cl S)\}$.
Set also 
$$\hat{\theta}_{\sub}(\cl S) = \inf\left\{\left\|\Phi^*(\sigma)^{-1}\right\| : 
\sigma \geq 0, \Tr(\sigma)\leq 1, \Phi^*(\sigma) \mbox{ inv. and }  \Phi\in \frak{C}_{\sub}(\cl S)\right\}.$$

\begin{proposition}\label{p_insub}
Let $H$ be a finite dimensional Hilbert space and $\cl S\subseteq \cl L(H)$ be an operator system.
Then $\thet_{\sub}(\cl S) = \thet(\cl S)$, 
$\theta_{\sub}(\cl S) = \theta(\cl S)$ and $\hat{\theta}_{\sub}(\cl S) = \hat{\theta}(\cl S)$.
\end{proposition}

\begin{proof}
Set $d = \dim(H)$. 
Since $\frak{C}(\cl S)\subseteq \frak{C}_{\sub}(\cl S)$, we have that $\thet_{\sub}(\cl S) \subseteq \thet(\cl S)$.
Suppose that $T\in \thet(\cl S)$ and let $\Phi \in \frak{C}_{\sub}(\cl S)$. 
Write $\Phi(S) = \sum_{i=1}^m A_i S A_i^*$, $S\in M_d$, where the operators $A_1,\dots,A_m \in \cl L(H,K)$
for some (finite dimensional) Hilbert space $K$ satisfy $\sum_{i=1}^m A_i^* A_i \leq I$.
Set $B_0 = \left(I - \sum_{i=1}^m A_i^* A_i\right)^{1/2}$. 
Let $\tilde{K}$ be a Hilbert space containing $K$, $V : H\to \tilde{K}$ be an isometry with 
range orthogonal to $K$, and $A_0 = VB_0$. Considering the operators $A_i$ as acting from $H$ into $\tilde{K}$, 
we have that 
$$\sum_{i=0}^m A_i^* A_i = B_0 V^* V B_0 + \sum_{i=1}^m A_i^* A_i = I.$$
In addition, $A_i^* A_0 = A_i^* V B_0 = 0$ and $A_0^* A_i = B_0 V^* A_i = 0$ for all $i = 1,\dots,m$, while
$A_0^* A_0 = B_0^2 \in \cl S$. Thus, the mapping $\tilde{\Phi} : \cl L(H)\to \cl L(\tilde{K})$, given by 
$\tilde{\Phi}(S) = \sum_{i=0}^m A_i S A_i^*$, is a quantum channel. 
By assumption, 
$\Phi(T) \leq \tilde{\Phi}(T)\leq I$, showing that $T\in \thet_{\sub}(\cl S)$.
It follows that 
$\thet_{\sub}(\cl S) = \thet(\cl S)$ and 
$\theta_{\sub}(\cl S) = \theta(\cl S)$.

Since $\frak{C}(\cl S)\subseteq \frak{C}_{\sub}(\cl S)$, we have that $\hat{\theta}_{\sub}(\cl S) \leq \hat{\theta}(\cl S)$.
For a completely positive trace decreasing map $\Phi$, let 
$$\eta(\Phi) = \min\left\{\|\Phi(\rho)\| : \rho \mbox{ a state}\right\}.$$
By the proof of Theorem \ref{th_eta}, 
$$\hat{\theta}_{\sub}(\cl S)^{-1} = \sup\left\{\eta(\Phi) : \Phi\in \frak{C}_{\sub}(\cl S)\right\}.$$
Fix $\epsilon > 0$ and choose $\Phi\in \frak{C}_{\sub}(\cl S)$ such that 
$\eta(\Phi) > \hat{\theta}_{\sub}(\cl S)^{-1} - \epsilon$. 
Let $A_1,\dots,A_m$ be Kraus operators for $\Phi$. 
Fix positive real numbers $p_1,\dots,p_l$ such that $\sum_{r=1}^l p_r = 1$ and 
$p_r \leq \eta(\Phi)$, $r = 1,\dots,l$, and 
a state $\rho$ on $H$.
Let $K_1,\dots,K_l$ be Hilbert spaces of dimension $d$, 
$\tilde{K} = \oplus_{r=1}^l K_r$ and $V_r : H \to K_r$ be a unitary operator, $r = 1,\dots,l$. 
With $B_0$ as in the first paragraph, write $B_r = V_r B_0$, $r = 1,\dots,l$, and let 
$\hat{\Phi} : \cl L(H)\to \cl L(K\oplus \tilde{K})$ be given by 
$$\hat{\Phi}(S) = \sum_{i=1}^m A_i S A_i^* + \sum_{r=1}^l p_r B_r S B_r^*, \ \ \ S\in \cl L(H),$$
where the operators $A_i$ and $B_r$ are considered as acting from $H$ into $K\oplus \tilde{K}$. 
A straightforward verification shows that $\hat{\Phi}$ is a quantum channel in $\frak{C}(\cl S)$. 
Moreover, 
$$\left\|\hat{\Phi}(\sigma)\right\| = \left\|\Phi(\sigma) \oplus \left(\oplus_{r=1}^l p_r B_r \sigma B_r^*\right)\right\| = \left\|\Phi(\sigma)\right\|$$
for every state $\sigma$, and hence
$\eta(\hat{\Phi}) = \eta(\Phi)$.
After letting $\epsilon \to 0$, we conclude that $\hat{\theta}(\cl S) \leq \hat{\theta}_{\sub}(\cl S)$.
\end{proof}

For a non-commutative graph $\cl S$ in $M_d$ and $k\in \bb{N}$, let 
$$\frak{C}\mbox{}_k(\cl S) = 
\left\{\Phi : \cl L(H) \to M_k \ : \ \Phi \mbox{ is a quantum channel with } \cl S_{\Phi} \subseteq \cl S\right\},$$
$$\thet\mbox{}_k(\cl S) = \{T\in M_d^+ : \Phi(T)\leq I \mbox{ whenever } \Phi\in \frak{C}_k(\cl S)\}$$
and
$$\frak{P}\mbox{}_k(\cl S) = \{\Phi^*(\sigma) : \Phi\in \frak{C}_k(\cl S), \sigma\geq 0, \Tr(\sigma)\leq 1\}.$$
As in Lemma \ref{l_dualn}, one can see that 
\begin{equation}\label{eq_ksh}
\frak{P}_k(\cl S)^{\sharp} = \thet\mbox{}_k(\cl S).
\end{equation} 
Let $\theta_k(\cl S) = \theta(\thet_k(\cl S))$ 
and $\hat{\theta}_k(\cl S)$ be defined as $\hat{\theta}(\cl S)$
but using quantum channels in $\frak{C}_k(\cl S)$. 
It is clear that $\frak{C}_k(\cl S)\subseteq \frak{C}_{k+1}(\cl S)$, $\thet_{k+1}(\cl S)\subseteq \thet_k(\cl S)$, 
$$\theta(\cl S)\leq \theta_{k+1}(\cl S) \leq \theta_k(\cl S), \ 
\hat{\theta}(\cl S) \leq \hat{\theta}_{k+1}(\cl S)\leq \hat{\theta}_k(\cl S), \ \ \ k\in \bb{N},$$ and that
$\hat{\theta}(\cl S) = \lim_{k\to\infty} \hat{\theta}_k(\cl S)$.
Since $\thet(\cl S) = \cap_{k\in \bb{N}}\thet_k(\cl S)$, we have that 
$\thet(\cl S) = \lim_{k\in \bb{N}}\thet_k(\cl S)$ and, by Lemma \ref{l_conth}, 
$\theta(\cl S) = \lim_{k\to\infty} \theta_k(\cl S)$.
We will shortly see that the sequence $(\theta_k(\cl S))_{k\in \bb{N}}$ stabilises.

\begin{lemma}\label{l_convcom}
Let $\cl S$ be a non-commutative graph in $M_d$,
$k\in \bb{N}$, and $\Phi \in \frak{C}_k(\cl S)$. 
Assume that $\Phi = t \Phi_1 + (1-t)\Phi_2$, where $\Phi_1, \Phi_2 : M_d\to M_k$ are quantum channels and $t\in (0,1)$.
Then $\Phi_1,\Phi_2\in \frak{C}_k(\cl S)$.
\end{lemma}

\begin{proof}
Suppose that $\{A_i\}_{i=1}^m$ and $\{B_p\}_{p=1}^l$ are families of Kraus operators for $\Phi_1$ and $\Phi_2$,
respectively. Then $\left\{\sqrt{t}A_i, \sqrt{1-t}B_p : i\in [m], p\in [l]\right\}$ is a family of Kraus operators for $\Phi$. 
Since $\cl S_{\Phi}$ is independent of the Kraus representation of $\Phi$, we have, in particular, that 
$A_i^*A_j \in \cl S$ and $B_p^*B_q \in \cl S$ for all $i,j\in [m]$ and all $p,q\in [l]$. 
Thus, $\cl S_{\Phi_1}\subseteq \cl S$ and $\cl S_{\Phi_2}\subseteq \cl S$.
\end{proof}

If $k\in \bb{N}$, let $\cl E_k$ be the set of all extreme points in the convex set 
of all quantum channels from $M_d$ to $M_k$.

\begin{proposition}\label{l_stable}
Let $\cl S$ be a non-commutative graph in $M_d$ and $k\geq d^2$. 
Then 
$$\thet\mbox{}_{d^2}(\cl S)  = \thet\mbox{}_k(\cl S) 
= \left\{T\in M_d^+ : \Psi(T)\leq I \mbox{ for all } \ \Psi\in \frak{C}_k(\cl S) \cap \cl E_k\right\}.$$
Thus, $\thet(\cl S) = \thet_{d^2}(\cl S)$ and $\theta(\cl S) = \theta_{d^2}(\cl S)$.
\end{proposition}

\begin{proof}
Suppose that $T\in M_d^+$ has the property that  
\begin{equation}\label{eq_ene}
\Psi(T) \leq I \ \mbox{ whenever } \ \Psi\in \frak{C}_k(\cl S) \cap \cl E_k.
\end{equation}
If $\Phi\in \frak{C}_k(\cl S)$, write $\Phi = \sum_{p=1}^l t_p \Phi_p$ as a convex combination, 
where $\Phi_p\in \cl E_k$, $p = 1,\dots,l$. 
By Lemma \ref{l_convcom}, $\Phi_p\in \frak{C}_k(\cl S)$. By assumption (\ref{eq_ene}), 
$$\Phi(T) = \sum_{p=1}^l t_p \Phi_p(T) \leq I.$$
Thus, 
\begin{equation}\label{eq_kext}
\thet\mbox{}_k(\cl S) = \left\{T\in M_d^+ : \Psi(T)\leq I \mbox{ for all } \ \Psi\in \frak{C}_k(\cl S) \cap \cl E_k\right\}.
\end{equation}

Fix $T\in \thet_{d^2}(\cl S)$, and
suppose that $\Psi\in \frak{C}_k(\cl S) \cap \cl E_k$. By \cite[Theorem 5]{choi}, 
there exists a Kraus representation 
$\Psi(S) = \sum_{i=1}^m A_i S A_i^*$, $S\in M_d$, such that the set $\{A_i^*A_j : i,j\in [m]\}$ is linearly independent. 
Thus, $m\leq d$. Let $P$ be the projection in $M_k$ onto the span of the ranges of the operators
$A_1,\dots,A_m$; then ${\rm rank}(P) \leq d^2$, and therefore $\Psi$ can be considered as a quantum channel 
into $M_{d^2}$. By assumption, $\Psi(T)\leq I$. 
By (\ref{eq_kext}), $T\in \thet_k(\cl S)$. 
Thus, $\thet_{d^2}(\cl S) \subseteq \thet_k(\cl S)$, and since 
$\thet_{k}(\cl S) \subseteq \thet_{d^2}(\cl S)$ trivially, we conclude that 
$\thet_{d^2}(\cl S) = \thet_k(\cl S)$. Thus,
$\thet(\cl S) = \thet_{d^2}(\cl S)$ and $\theta(\cl S) = \theta_{d^2}(\cl S)$.
\end{proof}

By Corollary \ref{c_ast} and Theorem \ref{p_theta1},
\begin{equation}\label{eq_1leq}
1\leq \theta(\cl S)\leq \hat{\theta}(\cl S) 
\end{equation}
for any non-commutative graph $\cl S$. 
We now use Proposition \ref{l_stable} to characterise when equalities hold in (\ref{eq_1leq}).

\begin{proposition}\label{p_1c}
Let $\cl S\subseteq M_d$ be an operator system. The following are equivalent:

(i) \ \ $\theta(\cl S) = 1$;

(ii) \ $\hat{\theta}(\cl S) = 1$;

(iii) $\cl S = M_d$.
\end{proposition}

\begin{proof}
(ii)$\Rightarrow$(i) follows from (\ref{eq_1leq}). 

(i)$\Rightarrow$(ii)
Suppose that $\hat{\theta}(\cl S) > 1$. 
Let $\delta \in \left(0,\frac{1}{d}\right)$ be such that $\hat{\theta}(\cl S)^{-1} <1-\delta$.
By Theorem \ref{th_eta}, for each $\Phi \in \mathfrak C(\cl S)$ there exists a state 
$\tau$ with $\| \Phi(\tau)\| <1- \delta$.  
Note that $\sigma:= \frac{1}{d-1} (I- \tau)$ is a state and 
$\frac{1}{d} I_d = \frac{1}{d} \tau + \left(1- \frac{1}{d}\right)\sigma$.  
Thus, 
$$\left\|\Phi\left(\frac{1}{d} I_d\right)\right\| \le \frac{1}{d} \| \Phi(\tau)\| + \left(1- \frac{1}{d}\right)\|\Phi( \sigma)\| 
< \frac{1-\delta}{d} +  \left(1- \frac{1}{d}\right) = 1- \frac{\delta}{d}.$$  
Setting $T = \frac{1}{d - \delta}I_d$, we have that $T \ge 0$ and $\|\Phi(T)\| \le 1$ for all $\Phi \in \mathfrak C(\cl S)$;
thus, $T \in \thet(\cl S)$
and so $\theta (\cl S) \ge \Tr (T) = \frac{d}{d - \delta} >1$. 

(iii)$\Rightarrow$(i) is straightforward. 

(i)$\Rightarrow$(iii) 
Set $k = d^2$. 
Following the proof of Theorem \ref{th_eta} (ii), one can see that 
\begin{equation}\label{eq_thk}
\theta\mbox{}_{k}(\cl S) =  \sup\left\{\inf\left\{\|\Phi(\rho)\|^{-1} : \Phi \in \mathfrak C_k(\cl S)\right\} : \rho \mbox{ a state}\right\}.
\end{equation}
For each quantum channel $\Phi : M_d\to M_k$ and each state $\rho$, we have 
$\|\Phi(\rho)\|^{-1} \geq 1$; thus, 
(\ref{eq_thk}) and Proposition \ref{l_stable} imply that
$$\inf\left\{\left\|\Phi(\rho)\right\|^{-1} : \Phi \in \mathfrak C_k(\cl S)\right\} = 1, \ \ \mbox{ for each state } \rho \mbox{ on } \bb{C}^d.$$
By Lemma \ref{l_closed} below, there exists $\Phi\in \mathfrak C_k(\cl S)$ such that 
$\left\|\Phi\left(\frac{1}{d}I_d\right)\right\| = 1$. 
It follows that there exists a unit vector $u\in \bb{C}^d$ such that $\Phi\left(\frac{1}{d}I_d\right) = uu^*$. 
The fact that the state $uu^*$ is pure now implies that $\Phi(\sigma) = uu^*$ for every pure state, and hence for every
state $\sigma$, on $\bb{C}^d$. 
Thus, $\Phi(T) = \sum_{i=1}^d (ue_i^*) T (e_i u^*)$, $T\in M_d$, and so 
$\cl S_{\Phi} = M_d$; since $\Phi\in \mathfrak C_k(\cl S)$, 
we have that $\cl S = M_d$.
\end{proof}

\noindent {\bf Remark }
Let $\cl S\subseteq M_d$ be a non-commutative graph. 
In view of Theorem \ref{th_sand} and
the fact that $\cp(\cl S)^{\sharp} \subseteq \fp(\cl S)^{\sharp}$, it is natural to ask if
the stronger inclusion $\thet(\cl S)\subseteq \cp(\cl S)^{\sharp}$ holds. 
The answer to this question is negative; indeed, 
$\{e_1,e_2\}$ is a clique for $\cl S_2$ and thus $I\in \cp(\cl S_2)$. Hence
$\cp(\cl S_2)^{\sharp} \subseteq \{T\in M_2^+ : \Tr(T) \leq 1\}$.
On the other hand, $e_1e_1^*$ and $e_2e_2^*$ are $\cl S_2$-abelian projections and so, 
by Theorem \ref{th_apcp}, $\cp(\cl S_2)^{\sharp} = \{T\in M_2^+ : \Tr(T) \leq 1\}$. 
It follows that $\kappa(\cl S_2) = 1$. 
On the other hand, by Proposition \ref{p_1c}, $\theta(\cl S_2) > 1$.
It follows that $\thet(\cl S_2)\not\subseteq \cp(\cl S_2)^{\sharp}$.

\subsection{Continuity}\label{ss_con}
In this subsection, we establish some continuity properties and exhibit a bound on the output system 
required for computing $\theta(\cl S)$. 
We use a classical concept of convergence due to Kuratowski. 
Let $\cl X$ be a topological space. For a sequence $(F_n)_{n\in \bb{N}}$ of subsets of $\cl X$, 
set 
$$\liminf_{n\in \bb{N}} F_n = 
\left\{\lim\mbox{}_{n\to\infty} x_n : 
(x_n)_{n\in \bb{N}} \in \Pi_{n\in \bb{N}}  F_n
\mbox{ a convergent sequence}\right\}$$
and 
$$\limsup_{n\in \bb{N}} F_n = 
\left\{x :  \mbox{ a cluster point of a sequence }
(x_n)_{n\in \bb{N}} \in \Pi_{n\in \bb{N}}  F_n\right\}.$$
We say that the sequence $(F_n)_{n\in \bb{N}}$ converges to the subset $F\subseteq \cl X$, 
and write $F = \lim_{n\to\infty} F_n$, 
if $F = \liminf_{n\in \bb{N}} F_n = \limsup_{n\in \bb{N}} F_n$.

\begin{lemma}\label{l_conth}
Let $(\cl A_n)_{n\in \bb{N}}$ be a sequence of convex corners in $M_d$ such that $\cup_{n\in \bb{N}}\cl A_n$ is bounded.
Assume that $\lim_{n\to\infty} \cl A_n = \cl A$ for some convex corner $\cl A\subseteq M_d$.
Then $\theta(\cl A_n) \to_{n\to\infty} \theta(\cl A)$.
\end{lemma}

\begin{proof}
Suppose that a subsequence $(\theta(\cl A_{n_m}))_{m\in \bb{N}}$ converges to $\delta$. 
Let $A_m\in \cl A_{n_m}$ be such that $\theta(\cl A_{n_m}) = \Tr(A_m)$, $m\in \bb{N}$. 
We may assume, without loss of generality, that $A_m\to_{m\to\infty} A$ for some $A\in M_d$. 
By assumption, $A\in \cl A$ and hence $\theta(\cl A) \geq \Tr(A) = \delta$. 
Thus, $\limsup_{n\in \bb{N}} \theta(\cl A_n) \leq \theta(\cl A)$.

Let $A\in \cl A$ be such that $\Tr(A) = \theta(\cl A)$. By assumption, 
there exists a sequence $(A_n)_{n\in \bb{N}}$ such that $A = \lim_{n\to\infty} A_n$. 
Then $\theta(\cl A)  = \lim_{n\to\infty} \Tr(A_n) \leq \liminf_{n\in \bb{N}} \theta(\cl A_n)$.
\end{proof}

\begin{lemma}\label{l_anan}
Let $\cl A, \cl A_n \subseteq M_n^+$, $n\in \bb{N}$, be non-empty sets 
such that $\cup_{n\in \bb{N}}\cl A_n$ is bounded 
and $\limsup_{n\in \bb{N}} \cl A_n \subseteq \cl A$. 
Then $\cl A^{\sharp}\subseteq \liminf_{n\in \bb{N}} \cl A_n^{\sharp}$.
\end{lemma}

\begin{proof}
For a bounded set $\cl B\subseteq M_d^+$ and an operator $T\in M_d^+$, let 
$$\delta_{\cl B}(T) = \sup\{\langle T,B\rangle : B\in \cl B\}.$$
Suppose that $T\in \cl A^{\sharp}$ and that $T\not\in \liminf_{n\in \bb{N}}\cl A_n^{\sharp}$. 
After passing to a subsequence if necessary, we assume that there exists $\delta > 0$ such that
\begin{equation}\label{eq_dist}
\left\|T - B\right\| \geq \delta, \ \ \  B\in \cl A_n^{\sharp}, \ n\in \bb{N}.
\end{equation}
We consider two cases.

\smallskip

\noindent {\it Case 1.} $\liminf_{n\in \bb{N}}\delta_{\cl A_n}(T) \leq 1$. 
In this case, there exists an increasing sequence $(n_p)_{p\in \bb{N}} \subseteq \bb{N}$ such that 
$\delta_{\cl A_{n_p}}(T) < 1 + \frac{1}{p}$, $p\in \bb{N}$. 
Thus, $\delta_{\cl A_{n_p}}\left(\frac{p}{p+1}  T\right) < 1$, 
and hence $\frac{p}{p+1}  T \in \cl A_{n_p}^{\sharp}$, $p\in \bb{N}$. 
Since $\frac{p}{p+1}  T \to_{p\to\infty} T$, this contradicts (\ref{eq_dist}). 

\smallskip

\noindent {\it Case 2.} $\liminf_{n\in \bb{N}}\delta_{\cl A_n}(T) > 1$. 
In this case, there exists $c > 1$ and $n_0\in \bb{N}$ such that $\delta_{\cl A_n}(T) \geq c$ for all $n\geq n_0$. 
Thus, there exists $A_n\in \cl A_n$ such that $\langle T,A_n\rangle \geq \frac{c+1}{2}$, $n\geq n_0$. 
Let $A$ be a cluster point of the sequence $(A_n)_{n\geq n_0}$. 
By assumption, $A\in \cl A$. Thus
$$1 \geq \langle T,A\rangle = \lim_{n\to \infty} \langle T,A_n\rangle \geq \frac{c+1}{2},$$
a contradiction. 
\end{proof}

In the sequel, we consider the operator systems in $M_d$ as closed subsets of the topological space $M_d$.

\begin{lemma}\label{l_closed}
Let $H$ be a finite dimensional Hilbert space, 
$\cl S, \cl S_n\subseteq \cl L(H)$, $n\in \bb{N}$, be operator systems and $k\in \bb{N}$. 
If $\limsup_{n\in \bb{N}} \cl S_n \subseteq \cl S$ then $\limsup_{n\in \bb{N}}$ $\frak{C}_k(\cl S_n)\subseteq \frak{C}_k(\cl S)$.
\end{lemma}

\begin{proof}
Without loss of generality, assume that 
$(\Phi_n)_{n\in \bb{N}} \subseteq \frak{C}_k(\cl S_n)$ is a sequence, and $\Phi$ is a quantum channel, 
such that 
$\Phi_n\to_{k\to\infty} \Phi$. 
Let $\tilde{\Phi}_n$ (resp. $\tilde{\Phi}$) be the complementary channel 
\cite{h} of $\Phi_n$
(resp. $\Phi$), $n\in \bb{N}$, acting from $\cl L(H)$ into $\cl L(\tilde{K})$, for some Hilbert space $\tilde{K}$
that can be chosen to be independent of $n$. 
By \cite{dsw}, $\cl S_{\Phi} = {\rm ran}(\tilde{\Phi}^*)$ and $\cl S_{\Phi_n} = {\rm ran}(\tilde{\Phi}_n^*)$, $n\in \bb{N}$.
By \cite{ksw}, 
$$\|\tilde{\Phi}_n^* - \tilde{\Phi}^*\| \to_{n\to\infty} 0.$$
Thus, if $R\in \cl L(\tilde{K})$ then $\tilde{\Phi}^*(R) = \lim_{n\to\infty} \tilde{\Phi}_n^*(R)$.
Since $\Phi_n \in \frak{C}_k(\cl S_n)$ for each $n$, 
we have $\tilde{\Phi}^*(R) \in \cl S$ and therefore $\cl S_{\Phi} \subseteq \cl S$.
\end{proof}

\begin{theorem}\label{th_contthet}
Let $k\in \bb{N}$ and 
$\cl S, \cl S_n$, $n\in \bb{N}$, be non-commutative graphs in $M_d$ such that 
$\cl S = \lim_{n\to\infty} \cl S_n$.
Then $\thet(\cl S) = \lim_{n\to\infty} \thet(\cl S_n)$ and $\theta(\cl S_n)$ $\to_{n\to \infty} \theta(\cl S)$.
\end{theorem}

\begin{proof}
Set $k = d^2$.
Suppose that $T_n\in \thet_k(\cl S_n)$, $n\in \bb{N}$, and $T_n\to T$ for some $T\in M_d$. 
Let $\Phi\in \frak{C}_k(\cl S)\cap \cl E_k$, and write $\Phi(S) = \sum_{i=1}^m A_i S A_i^*$ in a Kraus representation
with $m\leq d$. 
Let $V = (A_1,\dots,A_m)$ be the corresponding row operator and set 
$B = V^* V = (A_i^* A_j)_{i,j=1}^m$; then $B\in M_m(\cl S)^+$. 
Note that $(V^{\rm t})^* V^{\rm t} = I$.

Since $M_m(\cl S) \subseteq \liminf_{n\to\infty} M_m(\cl S_n)$, there exist $B_n\in M_m(\cl S_n)$, $n\in \bb{N}$,
such that $B_n\to_{n\to\infty} B$. We can moreover assume that $B_n = B_n^*$, $n\in \bb{N}$. 
Since $B\geq 0$, there exists a sequence $(\delta_n)_{n\in \bb{N}}\subseteq \bb{R}^+$ with $\delta_n\to_{n\to\infty} 0$
such that $B_n + \delta_n I \geq 0$, $n\in \bb{N}$. Thus, we may assume that $B_n\geq 0$, $n\in \bb{N}$. 

Since $dm\leq k$, there exists an isometry $W : \bb{C}^{md}\to \bb{C}^k$ 
such that $V = WB^{1/2}$. Let $V_n = WB_n^{1/2}$, $n\in \bb{N}$. 
Then $V_n^* V_n = B_n$, $n\in \bb{N}$, and $V_n\to_{n\to\infty} V$. 
Since $\|V^{\rm t}\| = 1$, we have that $\|V_n^{\rm t}\|\to_{n\to\infty} 1$. 
Letting $\tilde{V}_n = \frac{1}{\|V_n^{\rm t}\|} V_n$, $n\in \bb{N}$, we thus have 
that $\tilde{V}_n\to_{n\to\infty} V$, 
$(\tilde{V}_n^{\rm t})^* \tilde{V}_n^{\rm t} \leq I$
and $\tilde{V}_n^* \tilde{V}_n \in M_m(\cl S_n)$, $n\in \bb{N}$. 

Write $\tilde{V}_n = (A_{n,1},\dots,A_{n,m})$, where $A_{n,i} : \bb{C}^d\to\bb{C}^k$, $i = 1,\dots,m$.
Then the map $\Phi_n : M_d\to M_k$, given by $\Phi_n(S) = \sum_{i=1}^m A_{n,i} S A_{n,i}^*$, $S\in M_d$,
is a subchannel. 
Moreover, $\|\Phi_n - \Phi\|_{\rm cb} \to_{n\to\infty} 0$. 
By Proposition \ref{p_insub}, $\Phi_n(T_n)\leq I$, $n\in \bb{N}$. After passing to a limit, we conclude that 
$\Phi(T)\leq I$, and hence $T\in \thet_k(\cl S)$. We thus showed that 
$\limsup_{n\in \bb{N}} \thet_k(\cl S_n)\subseteq \thet_k(\cl S)$.

Suppose that $\Phi_n\in \frak{C}_k(\cl S_n)$ and $\sigma_n\in M_k^+$, $\Tr(\sigma_n) \leq 1$, 
are such that $\Phi_n^*(\sigma_n) \to_{n\to \infty} A$, for some $A\in M_d^+$. 
Assume, without loss of generality, that 
$\Phi_n \to_{n\to \infty} \Phi$ and $\sigma_n \to_{n\to \infty} \sigma$. 
By Lemma \ref{l_closed}, $\Phi\in \frak{C}_k(\cl S)$.
Moreover, $\Phi_n^*(\sigma_n)\to \Phi^*(\sigma)$. 
We thus showed that $\limsup_{n\in \bb{N}} \frak{P}_k(\cl S_n) \subseteq \frak{P}_k(\cl S)$.
Lemma \ref{l_anan} and identity (\ref{eq_ksh})
imply that $\thet_k(\cl S)\subseteq \liminf_{n\in \bb{N}} \thet_k(\cl S_n)$. 
Proposition \ref{l_stable}  now implies that $\thet(\cl S) = \lim_{n\to\infty} \thet(\cl S_n)$.
It is now straightforward to show that $\theta(\cl S_n)\to_{n\to \infty} \theta(\cl S)$.
\end{proof}

In the next corollary, we denote by $M_d^0$ the real vector space of all 
non-zero hermitian matrices of trace zero.

\begin{corollary}\label{c_Lamb}
The function 
$\Lambda \to \theta(\{\Lambda\}^{\perp})$ from $M_d^0$ into $\bb{R}^+$ is continuous
and has range $[\delta_1,\delta_2]$ for some $1 < \delta_1 < \delta_2\leq d$. 
\end{corollary}

\begin{proof}
We can clearly assume that the operators $\Lambda$ have norm one. 
Suppose that $\Lambda_n\to_{n\to\infty} \Lambda$, and consider $\Lambda_n\Lambda_n^*$ and $\Lambda\Lambda^*$
as projections on $M_d$.
Then $\Lambda_n\Lambda_n^*$ $\to_{n\to\infty} \Lambda\Lambda^*$; thus,
$(\Lambda_n\Lambda_n^*)^{\perp}$ $\to_{n\to\infty} (\Lambda\Lambda^*)^{\perp}$ and, by
 \cite{halmos}, $\lim_{n\to\infty} \{\Lambda_n\}^{\perp} = \{\Lambda\}^{\perp}$
as subspaces of $M_d$. 
By Theorem \ref{th_contthet}, $\theta(\{\Lambda_n\}^{\perp})$ $\to_{n\to\infty}$ $\theta(\{\Lambda\}^{\perp})$.

Since the domain of the function under consideration is connected and compact, its range is a closed interval 
$[\delta_1,\delta_2]$. 
The fact that $\delta_1 > 1$ follows from Proposition \ref{p_1c}. 
\end{proof}


\section{Open questions}\label{s_oq}

In this section, we discuss some open questions, arising naturally from the previous results. 

\begin{question}\label{q_tth}
Does the equality $\theta(\cl S) = \hat{\theta}(\cl S)$ hold true for every non-commutative graph $\cl S$?
\end{question}

This is perhaps the most fundamental open question about the parameters we have introduced.
In view of Theorem \ref{th_eta}, such an equality amounts to 
exchanging the order of the infimum and the supremum in its statement. 
We note that standard minimax theorems 
do not apply in any obvious way. 
The question is related to the possibility to lift the duality theory implicit in Lov\'asz original work \cite{lo}
and developed in \cite{gls} (see also \cite{gls_book} and \cite{knuth}), 
leading to several equivalent characterisations in the commutative case. 
In particular, it would be of
interest to study weighted versions of the parameters $\theta$ and $\hat{\theta}$, 
and establish a non-commutative version of the classical result from \cite{lo} 
stating that, for any graph $G$, we have $\thab(G)^{\flat} = \thab(G^c)$. 
Such an approach will be based on examining the following question:

\begin{question}
Does the parameter $\hat{\theta}$ arise from a convex corner?
\end{question}

We were able to establish the continuity of $\theta$ by exhibiting a bound on the 
size of the output system. 
We are not aware if a similar approach is possible for the case of $\hat{\theta}$:

\begin{question}
Given $d\in \bb{N}$, does there exist $k \in \bb{N}$ (depending on $d$), 
such that, for every non-commutative graph $\cl S\subseteq M_d$,  
the parameter $\hat{\theta}(\cl S)$ can be computed using channels $\Phi : M_d\to M_k$?
\end{question}

\begin{question}
Is the map $\cl S \longrightarrow \hat{\theta}(\cl S)$ continuous?
\end{question}

While we established the submultiplicativity of $\hat{\theta}$, leading to a bound on the 
Shannon capacity of a non-commutative graph, we do not know whether similar bounds can be 
formulated in terms of other parameters. In particular, we ask:

\begin{question}
Is the parameter $\nph$ submultiplicative?
\end{question}

In Proposition \ref{Snspaces} we identified most of the introduced parameters in the case of 
the non-commutative graph $\cl S_n$. However, we do not know the value of the Lov\'{a}sz numbers for this 
operator system (even in the case where $n = 2$):

\begin{question}\label{q_sn}
What are the values of $\theta(\cl S_n)$ and $\hat{\theta}(\cl S_n)$? 
\end{question}

Finally, it would be of interest to find a more precise version of Corollary \ref{c_Lamb}:

\begin{question} 
What are the precise values of $\delta_1$ and $\delta_2$ in Corollary \ref{c_Lamb}?
\end{question}

\end{document}